\documentclass[12pt]{amsart}
\usepackage{amsfonts,amssymb,amscd,amstext,mathrsfs}
\usepackage[utf8]{inputenc} 
\usepackage{hyperref}
\usepackage{verbatim}

\usepackage{graphics}
\usepackage{graphicx}

\usepackage{times}
\usepackage{enumerate}
\usepackage[up,bf]{caption}
\usepackage{color}

\input xy
\xyoption{all}

\newtheorem{theorem}{Theorem}[section]

\newtheorem{lemma}[theorem]{Lemma}
\newtheorem{corollary}[theorem]{Corollary}

\theoremstyle{definition}

\newtheorem{remark}[theorem]{Remark}

\numberwithin{equation}{section}
\numberwithin{figure}{section}

\newcommand\Cscr{\mathscr{C}}

\newcommand\Oscr{\mathscr{O}}

\newcommand\B{\mathbb{B}}
\newcommand\C{\mathbb{C}}
\newcommand\D{\overline{\mathbb D}}
\newcommand\CP{\mathbb{CP}}

\renewcommand\D{\mathbb D}

\newcommand\N{\mathbb{N}}

\newcommand\R{\mathbb{R}}

\newcommand\Z{\mathbb{Z}}

\newcommand\igot{\mathfrak{i}}

\renewcommand\igot{\mathfrak{i}}

\newcommand\E{\mathrm{e}}

\renewcommand\imath{\igot}

\newcommand\hra{\hookrightarrow}

\newcommand\wt{\widetilde}
\newcommand\wh{\widehat}

\newcommand\dist{\mathrm{dist}}

\newcommand\length{\mathrm{length}}

\newcommand\Aut{\mathrm{Aut}}

\newcommand\Id{\mathrm{Id}}

\def\dist{\mathrm{dist}}

\def\length{\mathrm{length}}

\numberwithin{equation}{section}

\begin{document}

\title{Embedded complex curves in the affine plane}

\author{Antonio Alarc\'on \; and \; Franc Forstneri{\v c}}

\address{Antonio Alarc\'on, Departamento de Geometr\'{\i}a y Topolog\'{\i}a e Instituto de Matem\'aticas (IMAG), Universidad de Granada, Campus de Fuentenueva s/n, E--18071 Granada, Spain}
\email{alarcon@ugr.es}

\address{Franc Forstneri\v c, Faculty of Mathematics and Physics, University of Ljubljana, and Institute of Mathematics, Physics, and Mechanics, Jadranska 19, 1000 Ljubljana, Slovenia}
\email{franc.forstneric@fmf.uni-lj.si}



\begin{abstract}
This paper brings several contributions to the classical 
Forster--Bell--Narasimhan conjecture and the 
%
%
Yang problem concerning
the existence of proper, almost proper, and complete injective 
holomorphic immersions of open Riemann surfaces in the affine plane $\C^2$
satisfying interpolation and hitting conditions.
We also show that every compact Riemann surface contains a Cantor set 
whose complement admits a proper holomorphic embedding in $\C^2$,
and every connected domain in $\C^2$ admits complete, 
everywhere dense, injectively immersed complex discs.
The focal point of the paper is a lemma saying for every compact bordered 
Riemann surface, $M$, closed discrete subset $E$
of $\mathring M=M\setminus bM$, and compact subset 
$K\subset \mathring M\setminus E$ without holes in $\mathring M$,
any $\Cscr^1$ embedding $f:M\hra\C^2$ which is holomorphic in $\mathring M$
can be approximated uniformly on $K$ by holomorphic embeddings
$F:M\hra\C^2$ which map $E\cup bM$ out of a given ball and satisfy
some interpolation conditions. 
\end{abstract}

\maketitle

%
%
%
%
\vspace*{-5mm}
\section{Introduction and Main Results}\label{sec:intro} 

%
%
This paper contributes to the following three interesting topics in global complex geometry, having the main focus on the interrelationship between them:
\begin{enumerate}[A]
\item[$\bullet$] The classical Forster--Bell--Narasimhan conjecture
(see  \cite{Forster1970,BellNarasimhan1990}) asking whether every open Riemann surface 
admits a {\em proper} holomorphic embedding in $\C^2$.  
The general case is still open; for positive results, 
see the surveys in \cite[Sections 9.10--9.11]{Forstneric2017E} 
and \cite{DiSalvoWold2022}. 

\item[$\bullet$] Yang's problem \cite{Yang1977AMS,Yang1977JDG} concerning the existence of
{\em complete} bounded complex submanifolds of $\C^n$; see the up-to-date comprehensive survey \cite{Alarcon2022Yang}.

\item[$\bullet$] The existence of {\em dense} holomorphic curves in complex manifolds; see Forstneri\v c and Winkelmann 
\cite{Winkelmann2005,ForstnericWinkelmann2005MRL} and \cite[Section 10]{AlarconForstneric2023Oka1}.
\end{enumerate}

The focal point of the paper is the following lemma, which is proved in Section \ref{sec:mainlemma}. 

%
%
\begin{lemma}\label{lem:main}
Let $M$ be a compact bordered Riemann surface with boundary of class $\Cscr^s$ for 
some $s>1$. Given a $\Cscr^1$ embedding $f:M\hra\C^2$ which is holomorphic on 
$\mathring M=M\setminus bM$, a compact set $K\subset \mathring M$ without holes, 
a compact polynomially convex set $L\subset \C^2$ such that 
$f(M\setminus \mathring K) \cap L=\varnothing$, 
finite sets $A=\{\alpha_1,\ldots,\alpha_l\}\subset \mathring M\setminus K$ 
and $B=\{\beta_1,\ldots,\beta_l\}\subset\C^2\setminus L$, 
a closed discrete set $E\subset \mathring M$ such that $E\cap (A\cup K)=\varnothing$,
and numbers $\epsilon>0$ and $r>0$, there is a holomorphic embedding $F:M\hra\C^2$ 
satisfying the following conditions:
\begin{enumerate}[\rm (a)]
\item $F(bM\cup E)\cap r\overline \B=\varnothing$. (Here, $\B$ denotes the unit ball of $\C^2$.)
\smallskip
\item $F(M\setminus \mathring K) \cap L=\varnothing$.
\smallskip
\item $\sup_{x\in K}|F(x)-f(x)| <\epsilon$.
\smallskip 
\item $F$ agrees with $f$ to a given finite order at a given 
finite set of points in $K\setminus f^{-1}(B)$.
\smallskip
\item $F(\alpha_j)=\beta_j$ for $j=1,\ldots,l$.
\end{enumerate}
\end{lemma}

We may assume that the surface $M$ in the lemma is a closed domain 
in a compact Riemann surface, $R$, whose boundary $bM$ 
consists of real analytic Jordan curves, and 
$M$ has no connected component without boundary. 
(See Stout \cite[Theorem 8.1]{Stout1965TAMS} and 
note that any conformal diffeomorphism of $M$ onto such a domain 
is of class $\Cscr^1(M)$; see \cite[Theorem 1.10.10]{AlarconForstnericLopez2021}.) 
A map $F:M\to\C^2$ is said to be holomorphic if it extends to a holomorphic 
map from an open neighbourhood of $M$ in the ambient Riemann surface.
A {\em hole} of a compact set in an open surface 
is a relatively compact connected component of its complement.
See Remark \ref{rem:openproblem} concerning the validity of the hypotheses of the lemma.

Lemma \ref{lem:main} is based on techniques developed by 
Wold \cite{Wold2006IJM,Wold2006MZ} and 
Forstneri\v c and Wold \cite{ForstnericWold2009} for constructing proper 
holomorphic embeddings of bordered Riemann surfaces in $\C^2$.
In their constructions, some boundary points of the given surface $M$
are sent to infinity where they remain at all subsequent steps. 
Our proof of Lemma \ref{lem:main} uses a similar construction 
with additional precision, but the points at infinity are finally brought back to $\C^2$. 
Thus, the main novelty of Lemma \ref{lem:main}, and of the related Lemma \ref{lem:mainbis},
is that we keep the entire Riemann surface $M$ as an embedded 
complex curve in $\C^2$ while pushing its boundary and the discrete set 
$E\subset \mathring M$ (or a countable family of discs in Lemma \ref{lem:mainbis})
arbitrarily far towards infinity. 

%
%
%
These two lemmas lead to new existence, approximation, interpolation, and hitting theorems for complete injectively immersed complex curves 
in $\C^2$, or in domains of $\C^2$, satisfying additional global conditions such as being proper, almost proper, or dense, which are presented in the sequel. 
At the same time, they give a simpler proof of a number of known results.
Our lemmas reduce proofs of these applications 
to formal induction schemes without the need of dealing with the technical issues 
at every step. A similar role is played by the Riemann--Hilbert method (see 
\cite{ForstnericGlobevnik2001MRL,DrinovecForstneric2007DMJ}
and \cite[Chapter 6]{AlarconForstnericLopez2021}), but when the target is a complex surface that technique, unlike ours,  tends to introduce self-intersections which cannot be removed since a generic immersion from a Riemann surface has transverse double points. 

As a first illustration of the constructions that can be carried out using Lemma \ref{lem:main},
we establish the following interpolation result for almost proper injective holomorphic 
immersions of bordered Riemann surfaces in $\C^2$. It is proved in Section \ref{sec:dense}.

%
%
\begin{theorem}\label{th:dense}
Let $M$ be a compact bordered Riemann surface with $\Cscr^s$ boundary for some $s >1$, 
and let $f: M\hra \C^2$ be a $\Cscr^1$ embedding which is holomorphic on $\mathring M$. 
Given a compact set $K\subset\mathring M$ without holes,  a discrete sequence $\alpha_j\in \mathring M\setminus K$ without repetition which clusters only on $bM$, a sequence
$\beta_j\in\C^2$ without repetition, and a number $\epsilon>0$, there is an almost proper injective holomorphic immersion 
$F: \mathring M\hra\C^2$ satisfying the following conditions:
\begin{enumerate}[\rm (i)]
\item $\sup_{x\in K}|F(x)-f(x)|<\epsilon$.
\smallskip 
\item $F$ agrees with $f$ to a given finite order at a given finite set in 
$K\setminus f^{-1}(\{\beta_j: j=1,2,\ldots\})$.\hspace*{-1mm}
\smallskip
\item $F(\alpha_j)=\beta_j$ for all $j=1,2,\ldots$.
\end{enumerate}
In particular, $F$ can be chosen such that $F(\mathring M)$ is every dense in $\C^2$.
 
If the sequence $\beta_j\in\C^2$ is closed and discrete, then there is 
a proper holomorphic embedding $F: \mathring M\hra\C^2$ satisfying 
conditions \rm{(i)-(iii)}.
\end{theorem}

Recall that a continuous map $f: X\to Y$ of topological spaces is said to be  
{\em almost proper} if for every compact set $K\subset Y$ the connected components 
of $f^{-1}(K)$ are all compact. Given a compact bordered Riemann surface $M$ as in 
Theorem \ref{th:dense}, one cannot hit an arbitrary countable subset of $\C^2$ by 
proper holomorphic maps $\mathring M\to\C^2$, but this can be done by almost 
proper maps. In fact, almost proper maps are in some sense the best class of 
holomorphic maps $\mathring M\to\C^2$ that can hit any given countable subset 
of $\C^2$.

The last part of Theorem \ref{th:dense} implies the
following known result concerning the Forster--Bell--Narasimhan Conjecture;
see Globevnik \cite{Globevnik2002} in the case of the disc, and
\cite[Corollaries 1.2 and 1.3]{ForstnericWold2009} by Forstneri\v c and Wold
and \cite[Theorem 1]{KutzschebauchLowWold2009} by Kutzschebauch et al.\ 
for an arbitrary $M$. We state it with additional precision concerning 
approximation and interpolation. 

%
%
\begin{corollary}\label{cor:main}
Given a holomorphic embedding $f:M\hra\C^2$ of a compact bordered 
Riemann surface $M$, a compact set $K\subset \mathring M$ without holes, and 
closed discrete sequences $\alpha_j\in \mathring M$
and $\beta_j\in\C^2$ without repetitions such that $K\cap\{\alpha_j:j\in\N\}=\varnothing$, 
we can approximate $f$ uniformly on $K$ by proper holomorphic embeddings 
$F:\mathring M\hra \C^2$ satisfying $F(\alpha_j)=\beta_j$ for all $j=1,2,\ldots$.
\end{corollary}

We give a unified proof of Theorem \ref{th:dense} and Corollary \ref{cor:main},
based on Lemma \ref{lem:main}, and we supply some details related to  
\cite[Lemma 2.2]{KutzschebauchLowWold2009}; see Remark \ref{rem:gap}. 
The analogous result for algebraic curves in $\C^2$ is a special case of 
\cite[Theorem 1.3]{ForstnericIvarssonKutzschebauchPrezelj2007}; 
see also \cite[Theorem 4.17.1]{Forstneric2017E}.

%
%

%
%
Recall that an immersed submanifold $\varphi: Z\to\R^n$ 
is said to be {\em complete} if the Riemannian metric on $Z$, obtained by 
pulling back the Euclidean metric on $\R^n$ via $\varphi$, is a complete metric; 
equivalently, for every proper path $\gamma:[0,1)\to Z$
the path $\varphi\circ\gamma:[0,1)\to\R^n$ has infinite Euclidean length.
It is obvious that every almost proper immersion $\varphi: Z\to\R^n$ is 
complete, any hence the immersions $F$ provided 
by Theorem \ref{th:dense} are complete. It seems that this gives the 
first examples of a specific bordered Riemann surface, other than the disc, 
admitting a complete nonproper injective holomorphic immersion into $\C^2$.
The construction of complete injectively immersed complex lines $\C\hra\C^2$
was given by the authors in \cite{AlarconForstneric2018PAMS}.

The technique of bringing back the points at infinity also applies in conjunction 
with \cite[Lemma 3.1]{ForstnericWold2013}, 
thereby yielding an analogue of Lemma \ref{lem:main} for circle domains 
with countably many boundary components in the Riemann sphere $\CP^1$; 
see Lemma \ref{lem:mainbis}. This gives a simpler proof of the theorem of 
Forstneri\v c and Wold \cite[Theorem 1.1]{ForstnericWold2013} saying that every 
circle domain in $\CP^1$ embeds properly holomorphically in $\C^2$; 
see Theorem \ref{th:FW}. (For domains with finitely many boundary components, 
this was proved by Globevnik and Stens\o ness \cite{GlobevnikStensones1995}.) 
As indicated in \cite[p.\ 500]{ForstnericWold2013}, 
the analogous result likely holds for circle domains in tori. 
Nothing seems to be known about this problem 
for domains in compact Riemann surfaces of genus $>1$,
where the main problem is to find a suitable initial embedding of the uniformized 
surface into $\C^2$.
On the other hand, in Section \ref{sec:countably-RS} we use  Lemma \ref{lem:main} 
to prove  the following hitting theorem for almost proper injective 
holomorphic immersions in $\C^2$ from domains obtained by removing
countably many pairwise disjoint closed discs from any compact Riemann surface. 

%
%
%
%
\begin{theorem}\label{th:countably-RS}
Let $R$ be a compact Riemann surface and $\Omega=R\setminus\bigcup_{i=0}^\infty D_i$ 
be an open domain in $R$ whose complement is the union of countably many pairwise
disjoint closed discs $D_i$ with $\Cscr^s$ boundaries for some $s>1$.  
Given a $\Cscr^1$ embedding $f:M_k=R\setminus\bigcup_{i=0}^k \mathring D_i\hra\C^2$ 
for some $k\ge 0$ that is holomorphic on the open bordered surface 
$\mathring M_k=R\setminus\bigcup_{i=0}^k D_i$, 
a compact set $K\subset \Omega$, a number $\epsilon>0$, and a countable set $B\subset\C^2$, there is an almost proper (hence complete) injective holomorphic 
immersion $F:\Omega \hra\C^2$ such that
\begin{enumerate}[\rm (i)]
\item $\sup_{x\in K}|F(x)-f(x)|<\epsilon$,
\item $F$ agrees with $f$ to a given finite order at a given finite set of points in 
$\Omega$, and
\item $B\subset F(\Omega)$.
\end{enumerate}
In particular, there exists an almost proper (hence complete) 
injective holomorphic immersion $\Omega\hra\C^2$ with everywhere dense image.
\end{theorem}

%
%

Lemma \ref{lem:main} can also be combined with the method developed by Forstneri\v c 
\cite{Forstneric2022RMI} for constructing complete bounded embedded holomorphic null 
curves in $\C^3$ with Cantor ends, as well as complete bounded minimal surfaces in 
$\R^3$ and some other related types of surfaces with Cantor ends. 
In this way we obtain the following result proved in Section \ref{sec:Cantor}.

%
%
\begin{theorem}\label{th:Cantor}
If $R$ is a compact Riemann surface and $B\subset\C^2$ is a countable subset, 
there exist a Cantor set $C\subset R$ and an almost proper injective 
holomorphic immersion $F: R\setminus C\hra\C^2$ whose image contains $B$. 
If $B$ is closed and discrete in $\C^2$ then $F$ can be chosen to be 
a proper holomorphic embedding.
Hence, every compact Riemann surface contains a Cantor set whose 
complement admits a proper holomorphic embedding in $\C^2$.
\end{theorem}

%
%

The Cantor sets which arise in the proof of Theorem \ref{th:Cantor}
are small modifications of the standard Cantor set in the plane, and they have 
almost full measure in a surrounding domain. 
The last statement in Theorem  \ref{th:Cantor} 
generalizes a recent result by Di Salvo and Wold
\cite[Theorem 1.1]{DiSalvoWold2022}, who constructed a Cantor set of large measure in 
$\CP^1$ whose complement admits a proper holomorphic embedding in $\C^2$. 
The first examples of Cantor sets in $\CP^1$ whose complements embed properly
in $\C^2$ were given by Orevkov \cite{Orevkov2008}, and 
Di Salvo \cite{DiSalvo2022} showed that Orevkov's construction 
also yields examples having Hausdorff dimension zero.

%
%
So far, we have been talking about (almost) proper injective holomorphic immersions 
in $\C^2$. However, Lemma \ref{lem:main} can also be applied to construct 
complete injectively immersed holomorphic curves in more general domains
in $\C^2$, at the cost of losing control of their conformal structure and in some
case of almost properness. 

%
%
%
To motivate this line of developments, we recall that 
Yang \cite{Yang1977AMS,Yang1977JDG} 
asked in 1977 whether there exist complete bounded complex submanifolds of a 
complex Euclidean space $\C^n$ of dimension $>1$. 
The Yang problem has been a focus of interest in the last decades; 
we refer to the recent survey \cite{Alarcon2022Yang}. It is an open problem 
whether for every compact bordered Riemann surface $M$ as in Lemma \ref{lem:main}
there is a complete holomorphic embedding $\mathring M\hra \C^2$ with bounded image; 
see \cite[Problem 1.5]{AlarconGlobevnik2017}. In fact, given such a Riemann surface 
$M$ other than the closed disc, all known complete holomorphic embeddings 
$\mathring M\hra\C^2$ are proper in $\C^2$ (see \cite[Corollary 1.2]{ForstnericWold2009}) 
or else the complex structure of the embedded surface may change.  
By using Lemma \ref{lem:main}, we construct complete embedded complex curves
%
%
with a given smooth structure in any pseudoconvex Runge domain of $\C^2$ 
as in the following theorem. 
In this case one clearly cannot control the complex structure of the examples. 

%
%
%
%
\begin{theorem}\label{th:pseudoconvex}
Let $D\subset\C^2$ be a pseudoconvex Runge domain and $B$ be a countable 
subset of $D$. On every open Riemann surface $S$ there are a domain $M$,  
which is 
%
%
%
diffeotopic to $S$, and a complete, almost proper, injective holomorphic 
immersion $F: M\hra D$ such that $B\subset F(M)$. If in addition the set $B$ 
is closed in $D$ and discrete, then $F: M\hra D$ can be chosen to be 
a complete proper holomorphic embedding.
\end{theorem}

Theorem \ref{th:pseudoconvex} is proved in Section \ref{sec:pseudoconvex}. 
The special case when $D=\C^2$ and $B=\varnothing$, guaranteeing the existence of properly embedded complex curves in $\C^2$ with arbitrary topology, was established 
in 2013 by Alarc\'on and L\'opez \cite[Theorem 4.5]{AlarconLopez2013JGEA}. 
This showed that there is no topological restriction to the  
Forster--Bell--Narasimhan conjecture. 
(For embeddings in $\C\times \C^*$ and $(\C^*)^2$, see 
Ritter \cite{Ritter2013JGEA,Ritter2018JRAM}, L\'arusson and Ritter \cite{LarussonRitter2014},
and Remark \ref{rem:otherdomains}.)
In the special case when $D=\B$ is the open unit ball and $B\subset\B$ is closed and discrete, Theorem \ref{th:pseudoconvex} was proved by Alarc\'on and Globevnik \cite{AlarconGlobevnik2017}. Likewise, when $D=\B$ or $D=\C^2$ and 
$S$ is of finite topology, it was established by the authors 
in \cite{AlarconForstneric2018PAMS}, except for the almost properness condition. 
For arbitrary pseudoconvex Runge domains $D$ in $\C^2$, Theorem \ref{th:pseudoconvex} also generalizes and simplifies the proofs of some hitting results for
(not necessarily complete) properly embedded complex curves in $D$ due to 
Forstneri\v c, Globevnik, and Stens\o ness \cite{ForstnericGlobevnikStensones1996}
and Alarc\'on \cite{Alarcon2020JAM}. Adapting the arguments in \cite{AlarconLopez2013JGEA,AlarconGlobevnik2017,AlarconForstneric2018PAMS} to the use of labyrinths of compact sets in pseudoconvex Runge domains, constructed by Charpentier and Kosi\'{n}ski in \cite{CharpentierKosinski2020}, leads to a proof of Theorem \ref{th:pseudoconvex} in the case when $B$ is closed in $D$ and discrete, 
or $S$ is finitely connected. The proof of Theorem \ref{th:pseudoconvex} that we give here,  
based on  Lemma \ref{lem:main} and using the labyrinths from \cite{CharpentierKosinski2020}, 
is considerably simpler and provides the general case of the theorem.

%
%
In Section \ref{sec:new} we establish the following analogue of Theorem 
\ref{th:pseudoconvex} in which we do not impose any condition 
whatsoever on the given connected domain in $\C^2$; the cost being not to guarantee 
almost properness of the obtained immersion.

%
%
\begin{theorem}\label{th:main}
Let $X\subset\C^2$ be a connected domain and $B$ be a countable 
subset of $X$. Given an open Riemann surface $S$, there are a domain $M\subset S$,
which is 
%
%
diffeotopic to $S$, and a complete injective holomorphic 
immersion $F: M\hra X$ such that $B\subset F(M)$. In particular, $F$ can be chosen 
to have everywhere dense image in $X$. 
\end{theorem}

All similar results in the literature pertain to special domains in $\C^2$;
see the discussion below Theorem \ref{th:pseudoconvex} and the survey \cite{Alarcon2022Yang}. On the other hand, omitting the injectivity condition 
in dimension two, it was shown by Forstneri\v c and Winkelman 
\cite{Winkelmann2005,ForstnericWinkelmann2005MRL} that every connected 
complex manifold $X$ with $\dim X >1$ admits an immersed holomorphic disc 
$\D=\{\zeta\in \C: |\zeta|<1\}\to X$ with dense image; if $\dim X>2$ then the immersion 
can be chosen injective. Recently the analogous result was obtained by the authors 
for any bordered Riemann surface and for some other classes of open Riemann surfaces \cite[Theorem 10.1]{AlarconForstneric2023Oka1}.

Let $d_{\rm H}$ denote the Hausdorff distance between subsets of Euclidean spaces.
The following corollary, which follows by inspecting the proof of Theorem \ref{th:main}, 
shows that every embedded holomorphic disc is arbitrarily close to a
complete one in the Hausdorff distance.

%
%
\begin{corollary}\label{co:main1}
Given a $\Cscr^1$ embedding $G:\overline\D\hra \C^2$ which is holomorphic on $\D$, 
a compact set $K\subset\D$, and a number $\epsilon>0$, there is a complete injective 
holomorphic immersion $F:\D\to \C^2$ such that $|F-G|<\epsilon$ on $K$ and 
$d_{\rm H}(\overline{F(\D)},G(\overline \D))<\epsilon$.
\end{corollary}

We do not know whether the injective immersion $F$ in Corollary \ref{co:main1} can be 
chosen to extend continuously to $\overline\D$ or to satisfy $|F-G|<\epsilon$ on $\D$. 
In particular, it remains an open question whether $F$ can be chosen such 
that $F(\D)$ is bounded by a Jordan curve. 
All these tasks can be carried out if one allows the map $F$ to have double points; see
\cite{AlarconDrinovecForstnericLopez2015PLMS,AlarconDrinovecForstnericLopez2019TAMS,Vrhovnik2023}.

Insisting on the almost properness condition, we also establish the following 
result which is obtained by a slight modification of the proof of Theorem \ref{th:main}. 
Again, we do not impose any condition whatsoever on the given connected domain 
$X$ in $\C^2$.

%
%
\begin{theorem}\label{th:main2}
Let $X\subset\C^2$ be a connected domain and $B$ be a countable 
subset of $X$. On every open Riemann surface $S$ there exists a connected, relatively compact 
domain $M$ such that $M$ has the same topological genus as $S$ and there is a complete almost proper, injective holomorphic immersion 
$F:M\to X$ with $B\subset F(M)$. In particular, $F$ can be chosen to have everywhere dense image in $X$.
\end{theorem}

Our method of proof does not allow to ensure that the domain $M$ in Theorem \ref{th:main2} is homeomorphic to the given open Riemann surface $S$. In particular, we cannot control its ends set, which could be more complicated than that of $S$.

%
%
\begin{remark}[On completeness]\label{rem:metric}
By a minor modification of the proofs, using that any two metrics on a compact space are comparable, we can ensure that the injective holomorphic immersions $F$ obtained in Theorems \ref{th:dense}, \ref{th:countably-RS}, \ref{th:Cantor}, \ref{th:pseudoconvex}, \ref{th:main}, and \ref{th:main2} are complete with respect to any given Riemannian metric (not necessarily complete or the Euclidean one) in the target domain $\C^2$, $D\subset\C^2$, or $X\subset \C^2$, respectively. 
\end{remark}

%
%
\begin{remark}[On the hypotheses in Lemma \ref{lem:main}] \label{rem:openproblem}
It is not known whether every compact bordered Riemann surface embeds 
holomorphically in $\C^2$. Here is a way to obtain such surfaces. 
Any compact Riemann surface, $R$, admits a holomorphic immersion $f:R\to \CP^2$ 
in the projective plane with finitely many simple double points $f(a_j)=f(b_j)$, 
where $a_j\ne b_j$ for $j=1,\ldots,m$ (see Griffiths and Harris \cite{GriffithsHarris1994}). 
Given a complex line $\Lambda\subset\CP^2$, the punctured Riemann surface  
$R'=R\setminus (f^{-1}(\Lambda) \cup\{b_1,\ldots,b_m\})$ is injectively immersed 
in $\CP^2\setminus \Lambda \cong \C^2$, and hence any compact 
domain in $R'$ is embedded in $\C^2$. 
There is considerable freedom in the above choices, showing that most
domains with smooth boundary in any compact Riemann surface satisfy 
Lemma \ref{lem:main}. However, we are not aware of suitable results in the literature
on controlling the location of double points in an immersed compact 
Riemann surface in $\CP^2$, and it seems an open problem whether one could put all 
punctures in the above construction in an arbitrarily small disc around any given point of $R$.
If this were true, then one could embed the interior of any finite bordered Riemann surface 
properly holomorphically in $\C^2$. 
\end{remark}

\section{Proof of Lemma \ref{lem:main}}\label{sec:mainlemma}

The proof involves four main steps: 
(1) using a holomorphic automorphism of $\C^2$ to satisfy the interpolation conditions in (e),
(2) exposing and sending to infinity a point in each boundary component of $M$,
(3) pushing the boundary $bM$ and the discrete set $E$ out of a given ball 
by a holomorphic automorphism of $\C^2$ (see condition (a)), 
and (4) bringing back the points at infinity.  
The first three steps are obtained by following and augmenting the proofs of 
\cite[Corollary 1.2]{ForstnericWold2009} and \cite[Lemma 2.2]{KutzschebauchLowWold2009}, 
while the last step uses a new idea. For the sake of readability we give a complete exposition, 
beginning with preliminaries.

%
%
Denote the coordinates on $\C^2$ by $z=(z_1,z_2)$, and let $\pi_i:\C^2\to \C$ 
for $i=1,2$ denote the projection $\pi_i(z_1,z_2)=z_i$. 
Let $\D$ be the open unit disc in $\C$ and $\B$ the open unit ball in $\C^2$.

Let $f:M\hra \C^2$ and the sets $B,L\subset\C^2$ be as in 
Lemma \ref{lem:main}, and let $c_1,\ldots, c_{l'} \in K\setminus f^{-1}(B)$ denote 
the points at which we must fulfil the interpolation condition (d).
By Mergelyan theorem (see \cite[Theorem 16]{FornaessForstnericWold2020}) we can 
approximate $f$ in the $\Cscr^1(M)$ topology by a holomorphic map 
$\tilde f:U\to\C^2$ from an open neighbourhood $U\subset R$ of $M$ which
agrees with $f$ to a given finite order $k$ at every point $c_1,\ldots,c_{l'}$. 
Assuming that the approximation is close enough and up to shrinking $U$ around $M$, 
we may assume that $\tilde f:U\hra\C^2$ is a holomorphic embedding satisfying 
$\tilde f(U \setminus \mathring K) \cap L=\varnothing$.  
By a standard transversality argument we can also ensure that 
$B\cap \tilde f(U) = \varnothing$. 
We replace $f$ by $\tilde f$ and drop the tilde. 

Finally, by a small perturbation of the map $f$, keeping the above conditions in place,
we can ensure that the set $L\cup f(M)$ is polynomially convex in $\C^2$. 
Indeed, by Stolzenberg's theorem \cite{Stolzenberg1963AM}
the polynomial hull of $L\cup f(bM)$ is the union of this set with complex
curves having their boundaries in $L\cup f(bM)$, and we can arrange 
that there are no such curves besides $f(M)$. 
Here is an explicit way of doing this. Choose a compact domain 
$M'\subset U$ containing $M$ in its interior such that $M$ is a strong deformation retract
of $M'$. We may assume that the function $f_1=\pi_1\circ f \in \Oscr(U)$ is nonconstant 
on each component of $U$. 
We approximate $f_2=\pi_2\circ f$ on $M$ (with interpolation at the points $c_1,\ldots,c_{l'}$) 
by a smooth function $\hat f_2$ on $M'$ which is holomorphic on $\mathring M'$ 
and does not extend holomorphically across any boundary point of $M'$. 
Set $\hat f=(f_1,\hat f_2)$. Take a point $p\in bM'$ at which $df_1(p)\ne 0$. 
Locally at the image point $q=\hat f(p)\in\C^2$ 
we can represent the complex curve $\Sigma'=\hat f (M')$ with smooth boundary 
$b\Sigma'=\hat f(bM')$ as a graph over 
the first coordinate such that the graphing function is holomorphic on the local projection
of $\Sigma'$ but does not extend holomorphically past the point $q_1=\pi_1(q)=f_1(p)$. 
It follows that $\Sigma'$ is not contained in any complex curve containing 
$q$ in the interior, since such a curve would provide a holomorphic extension of the 
graphing function to a neighbourhood of $q_1$. We claim that $L\cup \Sigma'$ 
is polynomially convex. Indeed, by Stolzenberg \cite{Stolzenberg1963AM}
the set $\wh{L\cup b\Sigma'}\setminus (L\cup b\Sigma')$ 
is a pure one-dimensional closed complex subvariety which is closed in 
$\C^2\setminus (L\cup b\Sigma')$. If this subvariety has an irreducible component
$\Lambda$ which is not contained in $\Sigma'$, then $\overline \Lambda$
must contain a connected component $C$ of $b\Sigma'$,
and the boundary uniqueness theorem 
(see Chirka \cite[Proposition 1, p.\ 258]{Chirka1989}) shows that 
$\Sigma' \cup \Lambda$ is a complex curve in $\C^2$ containing $C$,
contradicting the choice of $\hat f$. This proves the claim. 
Since $M$ is holomorphically convex in $M'$, Rossi's 
local maximum modulus principle  (see Rosay \cite{Rosay2006} for a simple proof) 
implies that $L\cup \hat f(M)$ is also polynomially convex.
Furthermore, for every compact set $K\subset \mathring M$ without holes 
such that $\hat f(M\setminus \mathring K)\cap L=\varnothing$ the set $L\cup \hat f(K)$
is polynomially convex. We now replace $f$ by a map satisfying all the stated conditions.

%
%
\medskip\noindent 
{\em (1) Fulfilling condition (e) in the lemma.} 
We have arranged above that the set $L\cup f(M)$ is polynomially convex in $\C^2$. 
Since $K$ has no holes in $\mathring M$ and $f(M\setminus \mathring K)\cap L=\varnothing$, 
the set 
\begin{equation}\label{eq:Lprime}
	L':= L\cup f(K)
\end{equation} 
is also polynomially convex (see the argument above). Recall that 
\[
	A=\{\alpha_1,\ldots,\alpha_l\}\subset \mathring M\setminus K
	\ \  \text{and}
	\ \ B=\{\beta_1,\ldots,\beta_l\}\subset\C^2\setminus L'.
\]
By the choice of $f$ we have that $B\cap f(U) = \varnothing$. 
By \cite[Proposition 4.15.3]{Forstneric2017E} there is a holomorphic automorphism
$\Phi\in\Aut(\C^2)$ which approximates the identity map on 
$L'$, it agrees with the identity to a given finite order $k$ at the points 
$f(c_1),\ldots,f(c_{l'}) \in f(K)\subset L'$, and it satisfies 
$\Phi(f(\alpha_j))=\beta_j$ for $j=1\ldots,l$. Replacing $f$ by $\Phi\circ f$ 
we may thus assume that $f$ fulfills condition (e) in the lemma
and the other properties remain in place.
We now add to $A$ the given finite set of points in $\mathring M$
at which we shall interpolate the map to a given finite order in the 
subsequent steps of the proof, and we suitably enlarge the set $K \subset \mathring M$ 
so that it has no holes and contains this new bigger set $A$,
while the discrete set $E\subset \mathring M$ remains in $\mathring M\setminus (A\cup K)$. 
This will ensure that the final map $F$ will satisfy condition (e).

%
%
\medskip\noindent 
{\em (2) Exposing boundary points.} 
We follow the exposition in \cite[Sect.\ 4]{ForstnericWold2009}. 
On each boundary curve $C_j\subset bM$ we choose a point $a_j$ 
and attach to $M$ a smooth embedded arc $\gamma_j\subset U$ such that 
$\gamma_j\cap M=\{a_j\}$, the intersection of $bM$ and $\gamma_j$ 
is transverse at $a_j$, and the arcs $\gamma_1,\ldots,\gamma_m$ 
are pairwise disjoint. Let $b_j$ denote the other endpoint of $\gamma_j$.
On the image side, we choose smoothly embedded pairwise disjoint arcs 
$
	\lambda_1,\ldots,\lambda_m\subset \C^2\setminus L',
$ 
where $L'$ is given by \eqref{eq:Lprime}, such that for every $j=1,\ldots, m$ we have that 
$\lambda_j\cap f(M)=f(a_j)$, $\lambda_j$ agrees with $f(\gamma_j)$ near the endpoint $f(a_j)$,
the other endpoint $p_j$ of $\lambda_j$ satisfies
\[
	|\pi_2(p_j)|> \sup\{|\pi_2(z)| : z\in L'\},
\]
the set $f(M) \cup \bigcup_{i=1}^m \lambda_i$ intersects the complex line
\begin{equation}\label{eq:Lambdaj}
	\Lambda_j=\pi_2^{-1}(\pi_2(p_j)) = \C\times \{\pi_2(p_j)\}
\end{equation} 
only at the point $p_j$, and the tangent vector to $\lambda_j$ at $p_j$ has nonvanishing 
second component. (See \cite[Fig.\ 2, p.\ 109]{ForstnericWold2009} where the 
projection $\pi_1$ is used in place of $\pi_2$.) 

We now modify $f$, keeping it fixed on a neighbourhood $U_1\subset U$ 
of $M$ and extending it to a smooth diffeomorphism $\gamma_j\to\lambda_j$ for every 
$j=1,\ldots,m$ such that $f(b_j)=p_j$. Set 
\[
	S=M\cup \bigcup_{j=1}^m \gamma_j \subset R.
\]
Applying Mergelyan theorem (see \cite[Theorem 16]{FornaessForstnericWold2020}), 
we can approximate $f$ as closely as desired in $\Cscr^1(S)$ by a holomorphic map 
$\tilde f:V\to\C^2$ on an open neighbourhood $V\subset R$ 
of $S$ such that $\tilde f$ agrees with $f$ to a given finite order $k$ at the points  
in the finite set $A \subset \mathring M$ defined in the previous step, 
and $\tilde f$ agrees with $f$ to the second order at the endpoints $a_j$ and $b_j$ 
of $\gamma_j$ for $j=1,\ldots,m$. 
If the approximation is close enough and up to shrinking $V$ around $S$, 
the map $\tilde f:V\hra\C^2$ is a holomorphic embedding satisfying 
\begin{equation}\label{eq:tildef-L}
	 \tilde f(V \setminus \mathring K) \cap L =\varnothing. 
\end{equation}
Furthermore, for every $j=1,\ldots,m$ the complex line $\Lambda_j$ \eqref{eq:Lambdaj} 
intersects the embedded complex curve $\tilde f(V)$ only at the point $p_j=\tilde f(b_j)=f(b_j)$ 
and the intersection is transverse.

We have now arrived at the main point of the exposing of points technique.
By \cite[Theorem 2.3]{ForstnericWold2009} there is a conformal diffeomorphism
\begin{equation}\label{eq:tau}
	\tau:M\to \tau(M)\subset V
\end{equation}
such that for every $j=1,\ldots, m$ we have that 
$\tau(a_j)=b_j$, $\tau$ maps a small neighbourhood $U_j\subset M$ of the point $a_j\in bM$ 
in a thin tube around the arc $\gamma_j$, $\tau$ agrees with the identity map to a given
order $k$ at the points of the finite set $A\subset \mathring M$, and $\tau$ is arbitrarily 
$\Cscr^1$ close to the identity map on $M\setminus \bigcup_{j=1}^m U_j$. 
We can choose $\tau$ to have any finite order of smoothness on $M$; 
for technical reasons which will become apparent in the next step 
we shall assume that it is of class $\Cscr^3(M)$. The map 
\begin{equation}\label{eq:h}	
	h=\tilde f\circ \tau:M\hra \C^2 
\end{equation}
is then a $\Cscr^3$ embedding which is holomorphic on $\mathring M$, 
its image $h(M)$ is a compact domain with $\Cscr^3$ boundary in the embedded 
complex curve $\tilde f(V)\subset\C^2$, and $\tau$ can be chosen such that 
for each $j=1,\ldots,m$ the complex line $\Lambda_j$ \eqref{eq:Lambdaj} 
intersects $h(M)$ only at $p_j=f(b_j)$. Assuming as we may that $\tau$ 
is close enough to the identity on $K$, \eqref{eq:tildef-L} implies
\[ 
	  h(M \setminus \mathring K) \cap L=\varnothing.	  
\] 
Let $g$ be a rational shear on $\C^2$ of the form 
\begin{equation}\label{eq:shear-g}
	g(z_1,z_2)=\left(z_1+ 
	\sum_{j=1}^m \frac{\rho\, \E^{\imath\theta_j}}{z_2-\pi_2(p_j)}+P(z_2), z_2\right)
\end{equation}
where $\rho>0$, $\theta_j\in \R$, and $P(z_2)$ is a holomorphic polynomial chosen such that 
the function $\sum_{j=1}^m \frac{\rho\, \E^{\imath\theta_j}}{z_2-\pi_2(p_j)}+P(z_2)$ vanishes
to order $k$ at the point $\pi_2(f(a))$ for every $a\in A$,   
and $P$ vanishes at every point $\pi_2(p_j)$ for $j=1,\ldots,m$. 
By taking the constant $\rho>0$ in \eqref{eq:shear-g} sufficiently small, 
the polynomial $P$ can be chosen such that $|P|$ is as small as desired 
on the compact set $\pi_2(L')$ where $L'$ is given by \eqref{eq:Lprime}. 
This gives a holomorphic embedding 
\begin{equation}\label{eq:gtildef}
	g\circ \tilde f:V\setminus\{b_1,\ldots,b_m\}\hra \C^2
\end{equation}
with simple poles at the points $b_1,\ldots, b_m$. Similarly, 
we have a $\Cscr^3$ embedding 
\begin{equation}\label{eq:gh}
	g\circ h = g\circ \tilde f\circ \tau: M'=M\setminus \{a_1,\ldots, a_m\} \hra \C^2
\end{equation}
which is holomorphic on $\mathring M$,  it approximates the embedding $h$ \eqref{eq:h} 
as closely as desired on $K$ provided that the constant $\rho>0$ in 
\eqref{eq:shear-g} is chosen small enough, it satisfies
\begin{equation}\label{eq:ghL}
	(g\circ h)(M'\setminus \mathring K) \cap L =\varnothing, 
\end{equation}
it agrees with $h$ (and hence with $f$) to order $k$ at the points of $A$, 
and $g\circ h$ sends the points $a_1,\ldots,a_m$ to infinity. 
More precisely, for every $j=1,\ldots,m$ the set
\begin{equation}\label{eq:sigmaj}
	\sigma_j=(g\circ h)(C_j\setminus \{a_j\}) \subset\C^2
\end{equation}
is a properly embedded curve of class $\Cscr^3$, diffeomorphic to $\R$, 
which is asymptotic to a line at every end (see \cite[Lemma 2]{Wold2006MZ} for the details), 
the first coordinate projection $\pi_1:\C^2\to \C$ maps  $\sigma_j$ to a proper curve 
$\tilde \sigma_j=\pi_1(\sigma_j) \subset \C$, and $\pi_1:\sigma_j\to \tilde \sigma_j$
is a diffeomorphism near infinity. Furthermore, the numbers $\theta_j\in \R$ in 
\eqref{eq:shear-g} can be chosen such that the projected curves $\tilde \sigma_j$ 
have different asymptotic directions, and for every sufficiently big number $s>0$ 
the set $\C\setminus \bigl(s\overline \D\, \cup\, \bigcup_{j=1}^m \tilde\sigma_j \bigr)$
has no bounded connected components. 
These choices are independent of the number $\rho>0$, which can be 
chosen arbitrarily small. 

With the curves $\sigma_j$ given by \eqref{eq:sigmaj} and the discrete set
$E\subset \mathring M$ as in the lemma, we define
\begin{equation}\label{eq:Gamma}
	\Gamma=\bigcup_{j=1}^m \sigma_j \subset\C^2\quad \text{and}\quad 
	E' = (g\circ h)(E)\subset\C^2\setminus\Gamma.   
\end{equation}
Since $E$ only clusters on $bM$, the set $E'$
only clusters on $\Gamma$, so $E'\cup\Gamma$
is closed and the projection $\pi_1:E'\cup\Gamma\to\C$ is proper.
Furthermore, since the curves $\sigma_j$ are asymptotic to lines at infinity,
for any $\C$-linear projection $\pi'_1:\C^2\to \C$ sufficiently close to $\pi_1$ 
the projection $\pi'_1:E'\cup\Gamma\to\C$ is still proper.
By a general position argument, $\pi'_1$ can be chosen such that
\begin{equation}\label{eq:injective}
	\pi'_1:E'\cup\Gamma\to\C\ \text{is proper},\  
	\pi'_1:E'\to\C\ \ \text{is injective, and}\  
	\pi'_1(E')\cap \pi'_1(\Gamma)=\varnothing.
\end{equation}
By a linear change of coordinates on $\C^2$ we may assume that this 
holds for $\pi_1$. Furthermore, if $\pi_1$ was not changed much, then
by our conditions on $\Gamma$ there is a number $s_0\ge 0$ such that 
\begin{equation}\label{eq:noholes}
	\text{the domain $\C \setminus (s\overline\D \cup \pi_1(\Gamma))$ 
	has no holes for $s \ge s_0$.}
\end{equation}

%
%
\begin{remark}\label{rem:gap}
The last two conditions in \eqref{eq:injective} are not discussed in 
\cite[proof of Lemma 2.2]{KutzschebauchLowWold2009}. 
Without them, we are unable to complete the proof
of Lemma \ref{lem:Wlem1} (the problem appears in the
construction of a shear $\psi$ in \eqref{eq:psi}). 
Indeed, we do not know how to prove \cite[Lemma 2.2]{KutzschebauchLowWold2009} 
without assuming that the linear projections $\C^2\to \C$ sufficiently close to $\pi_1$ 
are proper on $\Gamma$.
\end{remark}

%
%
\medskip \noindent 
{\em (3) Pushing $E'\cup\Gamma$ out of the ball $r\overline\B$.} 
We shall find an automorphism $\Phi\in\Aut(\C^2)$ sending the set $E'\cup\Gamma$  
out of the given ball $r \overline \B$. This is accomplished by the following lemma based 
on \cite[Lemma 1]{Wold2006MZ} and \cite[Lemma 2.2]{KutzschebauchLowWold2009}.
In light of Remark \ref{rem:gap}, we include a proof.

%
%
\begin{lemma}\label{lem:Wlem1} 
Let $E'$ and $\Gamma$ be as in \eqref{eq:Gamma}, satisfying conditions \eqref{eq:injective} 
and \eqref{eq:noholes} for the projection $\pi_1(z_1,z_2)=z_1$. Given a compact polynomially 
convex set $L\subset \C^2$ with $L\cap (E' \cup \Gamma) =\varnothing$ 
and numbers $r>0$ (big) and $\epsilon>0$ (small), there is an automorphism 
$\Phi\in \Aut(\C^2)$ satisfying the following conditions:
\begin{enumerate}[\rm (i)] 
\item $|\Phi(z)-z| < \epsilon$ for all $z\in L$,
\item $\Phi(E'\cup \Gamma)\subset \C^2\setminus r \overline \B$, and 
\item $\Phi$ agrees with the identity map to a given order $k$ at 
given points $q_1,\ldots, q_l\in L$.
\end{enumerate}
\end{lemma}

\begin{proof}
We shall obtain $\Phi$ as a composition $\Phi=\phi\circ\psi$ of 
two automorphisms, where $\phi$ will do the main job and $\psi$
will be a shear \eqref{eq:psi} taking care of things at infinity. 

Choose a compact polynomially convex set $L'\subset\C^2$ containing $L$ in its interior 
such that $L'\cap (E\cup \Gamma)=\varnothing$ and a number $\epsilon'>0$ to be
specified later. Let $s_0\ge 0$ be as in \eqref{eq:noholes}.
Pick $s\ge s_0$ such that $L' \subset s\D \times \C$ and set 
\begin{equation}\label{eq:tildeE}
	\wt E=E' \cap (s\overline\D\times \C) \quad\ \text{and}\quad
	 \wt \Gamma =\Gamma\cap (s\overline\D\times \C). 
\end{equation}
We move $\wt \Gamma$ out of the ball $r\overline\B$ by an isotopy of embeddings 
of class $\Cscr^3$ within the set $\C^2\setminus L'$. 
Since $\wt \Gamma$ is a union of curves,  the union of $L'$ with the image of 
$\wt \Gamma$ at every stage of the isotopy is polynomially convex
by Stolzenberg \cite{Stolzenberg1966AM}.
Hence, \cite[Theorem 2.1]{ForstnericLow1997} due to Forstneri\v c and L\o w 
furnishes an automorphism $\phi_1 \in\Aut(\C^2)$ satisfying
\begin{enumerate}[\rm (i')]
\item $|\phi_1(z)-z| < \epsilon'$ for all $z\in L'$, and 
\item $\phi_1(\wt \Gamma) \cap r \overline \B=\varnothing$.  
\end{enumerate}
(The proof of \cite[Theorem 2.1]{ForstnericLow1997} relies on the Anders\'en--Lempert
theorem in the version given by Forstneri\v c and Rosay \cite{ForstnericRosay1993}; 
see also \cite[Theorem 4.9.2]{Forstneric2017E}.) 

Since the discrete set $\wt E$ only clusters on $\wt \Gamma$ (see \eqref{eq:tildeE}), 
condition (i') implies that $\wt E=\wt E_1\cup \wt E_2$ where 
$\phi_1(\wt E_1) \cap r \overline \B=\varnothing$ and 
the set $\wt E_2=\wt E\setminus \wt E_1$ is finite. 
The compact set $\wt E_1\cup \wt \Gamma \cup L'$ is polynomially convex
by \cite[Lemma 2.3]{KutzschebauchLowWold2009}. 
Therefore, \cite[Proposition 4.15.3]{Forstneric2017E} furnishes an automorphism
$\phi_2 \in\Aut(\C^2)$ which is arbitrarily close to the identity map on 
$\phi_1(\wt E_1\cup \wt \Gamma \cup L')$ and maps the finite set 
$\phi_1(\wt E_2)$ into $\C^2\setminus  r \overline \B$. 
If the approximation is close enough then the automorphism 
$\phi_2\circ \phi_1\in\Aut(\C^2)$ satisfies the conditions 
\begin{enumerate}[\rm (i'')]
\item $|(\phi_2\circ \phi_1)(z)-z| < \epsilon'$ for all $z\in L'$,  and 
\item $(\phi_2\circ \phi_1)(\wt E\cup \wt \Gamma) \cap r \overline \B=\varnothing$.   
\end{enumerate}
We now correct $\phi_2\circ \phi_1$ so that the above conditions are preserved and 
the interpolation condition (iii) holds. Since $\phi_2\circ \phi_1$ is close to the identity 
on $L'$, its $k$-jet at each point $q_j$ is close to the $k$-jet of the identity map. 
To satisfy (iii) we take $\phi=\phi_3\circ \phi_2\circ \phi_1$ where a suitable 
automorphism $\phi_3 \in\Aut(\C^2)$ is 
obtained by \cite[proof of Theorem 4.9.2, p.\ 140]{Forstneric2017E}, which relies on 
the jet-interpolation theorem for holomorphic automorphisms 
\cite[Corollary 4.15.5, p.\ 174]{Forstneric2017E}.
Note that $\phi_3$ can be chosen arbitrarily close
to the identity map on any given compact set in $\C^2$ if the $k$-jet 
of $\phi_2\circ \phi_1$ at each point $q_j$ is close enough to the $k$-jet of the identity map.
Hence, assuming that $\epsilon'>0$ is chosen small enough, 
the automorphism $\phi=\phi_3\circ \phi_2\circ \phi_1$ satisfies condition (i), condition 
\begin{equation}\label{eq:condition-phi}
	\phi(\wt E \cup \wt \Gamma) \cap r \overline \B=\varnothing,
\end{equation} 
and condition (iii).  Recall the notation \eqref{eq:tildeE} and set 
\[
	E'' = E' \setminus \wt E\quad\text{and}\quad 
	\Gamma'=\Gamma\setminus \wt \Gamma.
\]
The problem now is that $\phi(E'' \cup \Gamma')$ may intersect the ball $r\overline \B$.
These intersections are removed by precomposing $\phi$ with a shear 
$\psi \in \Aut(\C^2)$ of the form 
\begin{equation}\label{eq:psi}
	\psi(z_1,z_2)=(z_1,z_2+\xi(z_1))
\end{equation}
for a suitably chosen entire function $\xi:\C\to\C$ which is close to $0$ 
on the disc $s\overline \D$. The idea is explained in 
\cite[proof of Lemma 2.2]{KutzschebauchLowWold2009} 
(based on \cite[Lemma 2.2]{BuzzardForstneric1997}); however,
some additions to their argument are necessary in light of Remark \ref{rem:gap}.
Note that $\phi^{-1}(r\overline \B)$ is a compact polynomially convex set.
Let $s\ge s_0$ be as above, and pick $s_1> s$ such that 
\begin{equation}\label{eq:s1}
	\pi_1(\phi^{-1}(r\overline \B))\subset s_1\D.
\end{equation}
The shear $\psi$ of the form \eqref{eq:psi} should be chosen such that
\begin{equation}\label{eq:condition-psi}
	\psi(E'\cup \Gamma)\cap \phi^{-1}(r\overline \B) =\varnothing.
\end{equation}
By \eqref{eq:s1} this does not impose any condition on the function $\xi$
on $\C\setminus s_1\D$, while on $s\overline \D$ it suffices 
to take $\xi$ close enough to $0$ in view of \eqref{eq:condition-phi}.
It is explained in \cite[Lemma 1]{Wold2006MZ} how to determine 
$\xi$ on the curves $\pi_1(\Gamma) \cap (s_1\D\setminus s\D)$
such that $\psi(\Gamma)\cap \phi^{-1}(r\overline \B)=\varnothing$. 
To find an entire function $\xi$ on $\C$ such that the above holds, 
one uses Mergelyan theorem (see \cite[Theorem 16]{FornaessForstnericWold2020}) 
and condition \eqref{eq:noholes}. Since the set $E'$ only clusters on $\Gamma$, 
there are at most finitely many points $Q=\{e_1,\ldots,e_i\}\subset E'$ such that 
$\psi(e_j)\in \phi^{-1}(r\overline \B)$ for $j=1,\ldots,i$,
i.e., condition \eqref{eq:condition-psi} fails only at these points. In view of
the last two conditions in \eqref{eq:injective} we can redefine $\xi$ at the points
$\pi_1(e_1),\ldots, \pi_1(e_i)\in s_1\D\setminus s\overline \D$, approximating 
the previously chosen function sufficiently closely on the 
polynomially convex set 
$
	\big[\pi_1((E'\setminus Q)\cup\Gamma) \cap (s_1\overline \D)\big] 
	\cup s\overline \D \subset\C,
$
so that the new shear map $\psi$ satisfies \eqref{eq:condition-psi}.
The interpolation of the identity at the given points $q_1,\ldots,q_l\in L$ 
is achieved by choosing $\xi$ such that it vanishes to order $k$ 
at every point $\pi_1(q_j)\in \pi_1(L)\subset s\D$ for $j=1,\ldots,l$. 
The automorphism $\Phi=\phi\circ\psi\in\Aut(\C^2)$ then satisfies Lemma \ref{lem:Wlem1}. 
\end{proof}

%
%
\smallskip \noindent 
{\em (4) Bringing back the points at infinity.}
Recall that $\tilde f:V\hra\C^2$ is a holomorphic embedding 
satisfying \eqref{eq:tildef-L} and $g\circ \tilde f:V\setminus \{b_1,\ldots,b_m\}\hra \C^2$
is given by \eqref{eq:gtildef}. Furthermore, $\tau:M\to \tau(M)\subset V$ is a conformal 
diffeomorphism in \eqref{eq:tau} and $h=\tilde f\circ\tau$ \eqref{eq:h}. 
The compact set $K\subset \mathring M$ is holomorphically convex in 
$\mathring M$, and in view of \eqref{eq:ghL} the compact set 
\begin{equation}\label{eq:tildeL}
	\wt L := L\cup (g\circ h)(K)\subset\C^2 
\end{equation}
is polynomially convex (for the details, see \cite{Wold2006IJM} or 
\cite[proof of Theorem 4.14.6, p.\ 168]{Forstneric2017E}).  
Let $\Phi\in\Aut(\C^2)$ be given by Lemma \ref{lem:Wlem1} with $L$ replaced 
by $\wt L$, the set $E'=(g\circ h)(E)$, and with interpolation to order $k$ 
(see condition (iii)) on the finite set $f(A)=(g\circ h)(A)\subset (g\circ h)(K)\subset \wt L$. 
Consider the holomorphic embedding 
\[
	\wt F := \Phi\circ g\circ \tilde f:V\setminus \{b_1,\ldots,b_m\}\hra \C^2.
\]
The image of $\wt F$ is an embedded complex curve in $\C^2$ containing 
the image of the embedding 
\[
	\wt F\circ \tau=\Phi\circ g\circ h: M\setminus \{a_1,\ldots,a_m\}\hra \C^2.
\]
By the construction, $\wt F\circ \tau$ satisfies Lemma \ref{lem:main}
except that the points $a_1,\ldots,a_m\in bM$ are sent to infinity. 
We now bring them back to $\C^2$ as follows. 
On every curve $\tau(C_j) \subset \tau(bM)$ we choose a proper closed arc 
$I_j\subsetneq \tau(C_j)$ containing the point $\tau(a_j)=b_j$ in its interior.
Let $v$ be a smooth vector field on $R$ along the set 
$I=\bigcup_{j=1}^m I_j$ which is transverse to $\tau(bM)$ and points to the interior of 
$\tau(M)$. Mergelyan theorem 
allows us to approximate $v$ by a holomorphic vector field 
on a neighbourhood of $\tau(M)$ in $V$, still denoted $v$,
which vanishes to order $k$ at every point of $A$.
(Note that the tangent bundle of $V$ is trivial, so we may think of $v$ as a function.) 
The flow $\psi_t$ of $v$ for small values of $|t|$, with $\psi_0=\Id$, 
exists on a neighbourhood
of $\tau(M)$ in $V$ and consists of biholomorphic maps whose $k$-jet at every point 
of $A$ agrees with the $k$-jet of the identity map. Since $v$ points to the interior of 
$\tau(M)$ along $I$, for small $t>0$ the closed domain $\psi_t(\tau(M))\subset V$ 
(which is conformally diffeomorphic to $M$) does not contain any of the points 
$b_1,\ldots,b_m$, and hence for such $t$ the embedding
\[
	F= \Phi\circ g\circ \tilde f \circ \psi_t \circ \tau : M\hra\C^2
\]
satisfies the conclusion of Lemma \ref{lem:main}.

%
%
\begin{remark} \label{rem:otherdomains}
(A) The same proof shows that Lemma \ref{lem:main} also holds
if the holomorphic map $f:M\to\C^2$ has finitely many branch points in 
$\mathring M$; see \cite[Theorem 1.1]{ForstnericWold2009} for the details. 

\smallskip

(B) The proof of Lemma \ref{lem:main} can be adapted to 
give an analogous result for embeddings of bordered Riemann surfaces in $\C\times \C^*$.
In this case, the compact set $L\subset \C\times \C^*$
should be holomorphically convex in $\C\times \C^*$, and condition (i) in the lemma
should be replaced by asking that $F(E\cup bM)$ lies outside the cylinder
$ 
	\{(z_1,z_2)\in \C^2: |z_1|\le r,\ 1/r\le |z_2|\le r\}
$ 
for a given $r>1$. An inductive application of this lemma yields the analogue of 
Corollary \ref{cor:main} for proper holomorphic embeddings $\mathring M\hra \C\times \C^*$
(see Ritter \cite[Theorem 4]{Ritter2013JGEA}).
\end{remark}


\section{Proof of Theorem \ref{th:dense}}\label{sec:dense}

We begin with some reductions. 
As in Lemma \ref{lem:main}, we assume as we may that $M$ is a closed 
smoothly bounded domain in a compact Riemann surface $R$. 
Fix $K$, $\alpha_j$, $\beta_j$, and $\epsilon$ as in the statement of Theorem \ref{th:dense}. 
Also, fix finitely many points $c_1,\ldots,c_l$ in 
$K\setminus f^{-1}(\{\beta_j: j=1,2,\ldots\})$ for 
the interpolation condition (ii), as well as a number $k\in\N$ for the interpolation order. 
By Mergelyan theorem in the $\Cscr^1$ topology with interpolation at the points $c_1,\ldots,c_l$
(see \cite[Theorem 16]{FornaessForstnericWold2020}), 
we may assume that $f$ is given by a holomorphic embedding $f:U\hra\C^2$ 
on an open neighbourhood $U\subset R$ of $M$.

Choose a smoothly bounded compact domain $K_0\subset \mathring M$ which 
is a strong deformation retract of $M$ such that $K\subset \mathring K_0$.  
By renumbering the points $\alpha_j$ and $\beta_j$, we may assume that 
$\alpha_1,\ldots, \alpha_{j}\in K_0$ and $\alpha_i\notin K_0$ for all $i>j$.
Applying Lemma \ref{lem:main} we find a holomorphic embedding 
$f_0:M\hra \C^2$ which approximates $f$ as closely as desired on $K$, 
it agrees with $f$ to order $k$ at the points $c_1,\ldots,c_l$, and it 
satisfies $f_0(\alpha_i)=\beta_i$ for $i=1,\ldots,j$. We now add the points
$\{\alpha_1,\ldots,\alpha_j\}$ to the finite set $\{c_1,\ldots,c_l\}$ and drop them from 
the list of $\alpha$-points in the theorem. 
Likewise, we drop $\beta_1,\ldots,\beta_j$ from the list of $\beta$-points.

We take $f_0$ as our new starting map and $K_0$ as the new initial compact set
in the theorem. Set $K_{-1}=\varnothing$ and $\epsilon_0=\epsilon/2$.
We may assume without loss of generality that $\beta_j\neq 0$ for all $j=1,2,\ldots$, 
and that there is $r_0>0$ such that
\begin{equation}\label{eq:inff0}
	f_0(M)\cap r_0\overline\B=\varnothing.
\end{equation}
Choose a sequence $0<r_0<r_1<r_2<\cdots$ with $\lim_{j\to\infty}r_j=+\infty$. 
Fix a point $\alpha_0\in K_0\setminus f_0^{-1}(\{\beta_j: j=1,2,\ldots\})$ 
and set $\beta_0=f_0(\alpha_0)$. We shall inductively construct a sequence 
of triples $X_j=\{f_j,K_j,\epsilon_j\}$, $j\in\N$, where
\begin{itemize}
\item $f_j: M\hra\C^2$ is a holomorphic embedding,
\smallskip
\item $K_j\subset\mathring M$ is a smoothly bounded compact domain which is a 
strong deformation retract of $M$, and
\smallskip
\item $\epsilon_j>0$ is a number,
\end{itemize}
such that
\begin{equation}\label{eq:exhaustion}
	\bigcup_{j\in\N} K_j=\mathring M
\end{equation}
and the following conditions hold for all $j\in\N$:
\begin{enumerate}[(1$_j$)]
\item $K_{j-1}\cup\{\alpha_i: i=0,\ldots,j\}\subset\mathring K_j$ and $\{\alpha_i: i>j\}\cap K_j=\varnothing$.
\smallskip
\item $\sup_{x\in K_{j-1}}|f_j(x)-f_{j-1}(x)|<\epsilon_j$.
\smallskip
\item $\epsilon_j<\epsilon_{j-1}/2$ and every holomorphic map $\varphi: \mathring M\to\C^2$ 
with $|\varphi-f_{j-1}|<2\epsilon_j$ everywhere on $K_{j-1}$ is an embedding on $K_{j-2}$.
\smallskip
\item $f_j(\alpha_i)=\beta_i$ for all $i\in\{0,\ldots,j\}$.
\smallskip
\item $f_j$ agrees with $f_{j-1}$ to order $k$ at $c_i$ for all $i\in\{1,\ldots,l\}$.
\smallskip
\item $f_j(M\setminus\mathring K_j)\cap r_j\overline\B=\varnothing$.
\smallskip
\item $f_j(M\setminus\mathring K_{j-1})\cap\min\{r_{j-1},|\beta_j|/2\}\overline\B=\varnothing$.
\end{enumerate}

Assuming the existence of such a sequence, the proof of Theorem \ref{th:dense} is completed as 
follows. Conditions (1$_j$), (2$_j$), (3$_j$), and \eqref{eq:exhaustion} ensure that there 
exists a limit map
\[
	F=\lim_{j\to\infty} f_j: \mathring M\to\C^2
\]
which is an injective holomorphic immersion and satisfies
\begin{equation}\label{eq:<2epsilonj}
	\sup_{x\in K_{j-1}}|F(x)-f_{j-1}(x)|<2\epsilon_j \quad \text{for all}\  j\in\N.
\end{equation}
This implies condition (i) in the theorem; recall that $f_0=f$ and $2\epsilon_1<\epsilon_0<\epsilon$. Conditions (4$_j$) and (5$_j$) ensure that $F(\alpha_i)=\beta_i$ for all $i=1,2,\ldots$ and $F$ agrees with $f$ to order $k$ at $c_i$ for all $i\in\{1,\ldots,l\}$; so, (ii) and (iii) hold as well. Now, \eqref{eq:<2epsilonj} and condition (6$_j$) guarantee that 
\begin{equation}\label{eq:almost-proper}
	\inf_{x\in bK_j}|F(x)|> r_j-2\epsilon_{j+1}>r_j-\epsilon_0\quad \text{for all}\  j\in\N.
\end{equation}
Since the increasing sequence of compact sets $K_j$ exhausts $\mathring M$ 
(see \eqref{eq:exhaustion}) and we have that $\lim_{j\to\infty}r_j=+\infty$, it follows 
that $F: \mathring M\hra\C^2$ is an almost proper map.
%
%

Likewise, \eqref{eq:<2epsilonj} and (7$_j$) give that 
\begin{equation}\label{eq:almost-proper}
	\inf_{x\in K_j\setminus\mathring K_{j-1}}|F(x)|> 
	\min\Big\{r_{j-1}\,,\,\frac{|\beta_j|}2\Big\}-\epsilon_0\quad \text{for all}\  j\in\N.
\end{equation}
If the sequence $\beta_j\in\C^2$ is closed and discrete, we have that 
$\lim_{j\to\infty}|\beta_j|=+\infty$, and hence
\[
	\lim_{j\to\infty}\min\Big\{r_{j-1}\,,\,\frac{|\beta_j|}2\Big\}=+\infty. 
\]
This, \eqref{eq:exhaustion}, \eqref{eq:almost-proper}, and conditions (1$_j$) imply that in this special case the injective immersion $F: \mathring M\hra \C^2$ which we constructed is a proper map, hence an embedding, thereby proving the final assertion 
of the theorem. 

Let us explain the induction. The basis is given by the triple $X_0=\{f_0,K_0,\epsilon_0\}$; observe 
that it meets conditions (1$_0$), (4$_0$), and (6$_0$), see \eqref{eq:inff0}, while (2$_0$), (3$_0$), 
(5$_0$), and (7$_0$) are void. For the inductive step we fix $j\in\N$ and assume that we have a 
triple $X_{j-1}=\{f_{j-1},K_{j-1},\epsilon_{j-1}\}$ enjoying (1$_{j-1}$), (4$_{j-1}$), and 
(6$_{j-1}$). By the Cauchy estimates, there is a number $\epsilon_j>0$ 
so small that (3$_j$) holds true. Recall that $\beta_j\neq 0$ and set 
\[
	L=\min\Big\{r_{j-1}\,,\,\frac{|\beta_j|}2\Big\} \,\overline\B.
\]
By (6$_{j-1}$), we have that
\[
	f_{j-1}(M\setminus \mathring K_{j-1})\cap L=\varnothing.
\] 
Choose a number $r>r_j$ so large that
\begin{equation}\label{eq:r>rj}
	f_{j-1}(M)\subset (r-\epsilon_j)\B.
\end{equation}
Lemma \ref{lem:main} then applies to the embedding $f_{j-1}$, the compact set $K_{j-1}$ (it has no holes since it is a strong deformation retract of $M$), the compact polynomially convex set $L\subset\C^2$, the singletons $\{\alpha_j\}\subset \mathring M\setminus K_{j-1}$ (see (1$_{j-1}$)) and $\{\beta_j\}\subset \C^2\setminus L$, the closed discrete set $E=\{\alpha_i: i>j\}\subset\mathring M$ (note that $E\cap (K_{j-1}\cup\{\alpha_j\})=\varnothing$ by (1$_{j-1}$)), and the numbers $\epsilon_j>0$ and $r>0$, furnishing us with a holomorphic embedding $f_j: M\hra \C^2$ satisfying (2$_j$), (5$_j$), (7$_j$), and the following conditions:
\begin{enumerate}[\rm (a)]
\item $f_j(bM\cup \{\alpha_i: i>j\})\cap r\overline\B=\varnothing$.
\smallskip
\item $f_j(\alpha_j)=\beta_j$ and $f_j(\alpha_i)=f_{j-1}(\alpha_i)$ for all $i\in\{1,\ldots,j-1\}$.
\smallskip
\item $f_j(K_{j-1})\subset r\B$.
\end{enumerate}
The first part of condition (b) is ensured by Lemma \ref{lem:main}-(e), while the second part is granted by Lemma \ref{lem:main}-(d). Condition (c) is implied by (2$_j$) and \eqref{eq:r>rj}.

Note that (b) and (4$_{j-1}$) ensure (4$_j$). Finally, in view of (a), (c), and the fact that $r>r_j$, we can choose a smoothly bounded compact domain $K_j\subset\mathring M$, being a strong deformation retract of $M$, such that (1$_j$) and (6$_j$) hold true. Indeed, by (a) we can first take a smoothly bounded compact domain $K_j'\subset\mathring M$ which is a strong deformation retract of $M$ such that $K_{j-1}\cup\{\alpha_i: i=0,\ldots,j\}=K_{j-1}\cup\{\alpha_j\}\subset\mathring K_j'$ and 
\begin{equation}\label{eq:Kj'}
	f_j(M\setminus\mathring K_j')\cap r\overline\B=\varnothing.
\end{equation} 
If the set $E'=E\cap K_j'=\{\alpha_i: i>j\}\cap K_j'$ is empty, then we simply choose $K_j=K_j'$. 
Otherwise, $E'$ is a finite set (recall that the sequence $\alpha_i\in\mathring M$ only clusters
on $bM$) and we can choose $K_j'$ so that $E'\subset\mathring K_j'\setminus K_{j-1}$; see
(1$_{j-1}$). Let $\Omega_1,\ldots,\Omega_m$ denote the connected components of
$K_j'\setminus \mathring K_{j-1}$; these are smoothly bounded compact annuli since 
$K_{j-1}\subset\mathring K_j'$ is a strong deformation retract of $K_j'$. Fix $i\in\{1,\ldots,m\}$. 
Conditions (a), (c), and \eqref{eq:Kj'} imply that the set 
$\Omega_i'=\Omega_i\cap f_j^{-1}(\C^2\setminus r\overline\B)$ is disjoint from $bK_{j-1}$ 
and it contains $E'\cap\Omega_i$ as well as an open neighborhood of $\Omega_i\cap bK_j'$. 
Since $\Omega_i'$ is open in $\Omega_i$, these properties and the maximum principle show 
that $\Omega_i'$ is path connected, and hence we can choose a smooth Jordan arc 
$\gamma_i\subset \Omega_i'\setminus\{\alpha_j\}$ containing $E'\cap\Omega_i$, having 
an endpoint in $\Omega_i\cap bK_j'$ and otherwise disjoint from $bK_j'$. Set 
$\gamma=\bigcup_{i=1}^m\gamma_i$. Note that 
$K_{j-1}\cup\{\alpha_j\}\subset \mathring K_j'\setminus \gamma$, $K_{j-1}$ is a 
strong deformation retract of $\mathring K_j'\setminus \gamma$, 
$\{\alpha_i: i>j\}\cap K_j'\setminus \gamma=\varnothing$, and 
$f_j(M\setminus (\mathring K_j'\setminus\gamma))\cap r\overline\B=\varnothing$. 
It is clear that every sufficiently large smoothly bounded compact domain 
$K_j\subset \mathring K_j'\setminus\gamma$ which is a strong deformation 
retract of $M$ satisfies conditions (1$_j$) and (6$_j$).
This closes the induction. 

Note that at each step of the induction we are allowed to choose the compact domain 
$K_j\subset\mathring M$ as large as desired under the only restriction imposed by the 
second part of condition (1$_j$). So, since the sequence 
$\alpha_j\in \mathring M$ is closed and discrete, we can proceed in such a way that 
condition \eqref{eq:exhaustion} is satisfied. 
This completes the proof of Theorem \ref{th:dense}.


\section{Proper embeddings of circle domains in $\CP^1$ into $\C^2$}
\label{sec:FW}

Recall that a circle domain in $\CP^1$ is an open domain of the form 
\begin{equation}\label{eq:circle}
	\Omega=\CP^1\setminus \bigcup_{i=0}^\infty D_i
\end{equation} 
where $D_i$ are pairwise disjoint closed round discs.
By the uniformization theorem of He and Schramm \cite{HeSchramm1993}, 
every domain of the form \eqref{eq:circle}, where $D_i$ are pairwise 
disjoint closed topological discs (homeomorphic images of $\overline\D$), 
is conformally equivalent to a circle domain. 
We give a simpler proof of the following result
\cite[Theorem 1.1]{ForstnericWold2013} due to Forstneri\v c and Wold. 

%
%
\begin{theorem} \label{th:FW}
Every circle domain in $\CP^1$ embeds properly holomorphically into $\C^2$. 
\end{theorem}

We shall use the following analogue of Lemma \ref{lem:main} adapted to this situation.

%
%
\begin{lemma}\label{lem:mainbis}
Let $\Omega$ be a circle domain \eqref{eq:circle} in $\CP^1$, and let 
$k\in \Z_+$. Given a $\Cscr^1$ embedding 
$f:M_k=\CP^1\setminus\bigcup_{i=0}^k \mathring D_i\hra\C^2$ 
which is holomorphic in $\mathring M_k$, 
a compact set $K\subset \Omega$ which is $\Oscr(\mathring M_k)$-convex, 
a compact polynomially convex set $L\subset \C^2$ such that 
$f(M_k\setminus \mathring K) \cap L=\varnothing$, points $\alpha\in \Omega\setminus K$
and $\beta\in\C^2\setminus L$, and a number $r>0$, we can approximate $f$ as closely 
as desired uniformly on $K$ by a holomorphic embedding $F:M_k\hra \C^2$ 
which agrees with $f$ at finitely many given points in $K$ and satisfies 
\[
	F\Big(bM_k \cup \bigcup_{i=k+1}^\infty D_i\Big) \subset \C^2\setminus r\overline \B,
	\quad 
	F(M_k\setminus \mathring K) \cap L=\varnothing,
	\ \ \ \text{and}\ \ \ 
	F(\alpha)=\beta. 
\]
\end{lemma} 

This lemma is obtained by combining \cite[proof of Lemma 3.1]{ForstnericWold2013} 
with the proof of Lemma \ref{lem:main} in Section \ref{sec:mainlemma}.
The only difference in \cite[Lemma 3.1]{ForstnericWold2013} when compared
to the technique used in the earlier papers
\cite{ForstnericWold2009,Wold2006IJM,Wold2006MZ} 
is that the conformal diffeomorphism $\tau:M\to \tau(M)\subset V$
in \eqref{eq:tau} is chosen such that it maps $M$ onto a domain
with piecewise smooth boundary in the ambient Riemann surface $V$.
(In the context of Lemma \ref{lem:mainbis}, we apply this argument 
to the bordered Riemann surface $M=M_k$.)
The finitely many corner points of $\tau(bM)$ are mapped to infinity
by the embedding $\tilde f$; see \eqref{eq:h}. The main point of this
change is to ensure that the first coordinate projection $\pi_1 :\C^2\to\C$,
restricted to the image of $M$, is injective near infinity;
this enables one to find a shear $\psi$ \eqref{eq:psi} in Step (3) of the proof of 
Lemma \ref{lem:main} such that the resulting automorphism $\Phi=\phi\circ\psi\in\Aut(\C^2)$
maps all the discs $D_i$ out of the given ball $r\overline\B$. 
It is clear that the method in step (4) of the proof of Lemma \ref{lem:main}, using
the flow $\psi_t$ of a suitably chosen holomorphic vector field on $V$, still 
applies to a domain with corners, so we can bring the points at infinity back to
$\C^2$. Finally, we approximate the new conformal diffeomorphism 
$\psi_t\circ \tau$ sufficiently closely by a smooth conformal diffeomorphism from $M$ 
onto its image in the ambient surface $V$, 
thereby removing the corners. We leave further details to the reader.

\begin{proof}[Proof of Theorem \ref{th:FW}] 
Let $\B$ denote the unit ball of $\C^2$. Set $M_0=\CP^1\setminus \mathring D_0$; 
this is a closed disc. Choose a holomorphic embedding $f_0:M_0\hra \C^2$ and 
a compact, smoothly bounded, $\Oscr(\Omega)$-convex set 
$K_0\subset \Omega$. Its $\Oscr(\mathring M_0)$-convex hull is the union of $K_0$ 
with at most finitely many smoothly bounded open discs in $\mathring M_0$, each 
containing a disc from the family $\{D_i\}_{i\in\N}$. Hence, there are finitely many discs
$D_{j(1)},\ldots,D_{j(k_1)}$ such that, setting 
$
	M_1=M_0\setminus \bigcup_{i=1}^{k_1}\mathring D_{j(i)},
$ 
the set $K_0$ is $\Oscr(\mathring M_1)$-convex. Let $J_1=\N\setminus \{j(1),j(2),\ldots,j(k_1)\}$.
By Lemma \ref{lem:mainbis} we can approximate the embedding $f_0|_{M_1}$ 
as closely as desired uniformly on $K_0$ 
by a holomorphic embedding $f_1:M_1\hra\C^2$ such that 
\[
	f_1\big(bM_1 \cup \bigcup_{i\in J_1} D_i\big) \subset \C^2\setminus \overline \B.
\]
Since $\mathring M_1 \setminus \bigcup_{i\in J_1} D_i =\Omega$, there is a 
compact, smoothly bounded, $\Oscr(\Omega)$-convex set $K_1\subset \Omega$
with $K_0\subset \mathring K_1$ such that 
$f_1(M_1\setminus \mathring K_1)\subset  \C^2\setminus \overline \B$. 
By the same argument as above, we can find finitely many discs 
$D_{j(k_1+1)},\ldots, D_{j(k_2)}$ from the given family such that, setting 
\[
	M_2=M_1\setminus \bigcup_{i=k_1+1}^{k_2}\mathring D_{j(i)}
	= M_0 \setminus \bigcup_{i=1}^{k_2}\mathring D_{j(i)},
\]
the set $K_1$ is $\Oscr(\mathring M_2)$-convex. Let $J_2=\N\setminus \{j(1),j(2),\ldots,j(k_2)\}$.
By Lemma \ref{lem:mainbis} applied to the embedding
$f_1|_{M_2}$ and the polynomially convex set $L=\overline\B\subset\C^2$ 
we can approximate $f_1$ on $K_1$ by a holomorphic 
embedding $f_2:M_2\hra\C^2$ such that 
\[
	f_2\big(bM_2 \cup \bigcup_{i\in J_2} D_i\big) \subset \C^2\setminus 2\overline \B
	\ \ \ \text{and} \ \ \ 
	f_2(M_2\setminus \mathring K_1) \subset \C^2\setminus \overline \B.
\]
Hence, there is a smoothly bounded $\Oscr(\Omega)$-convex set 
$K_2\subset \Omega$ with $K_1\subset \mathring K_2$ such that 
\[
	f_2(M_2\setminus \mathring K_2)\subset  \C^2\setminus 2\overline \B. 
\]

Continuing inductively, we obtain the following:
\begin{enumerate}[(1)]
\item an increasing sequence $K_0\subset K_1\subset K_2 \subset \cdots$
of compact, smoothly bounded, $\Oscr(\Omega)$-convex domains with 
$K_j\subset \mathring K_{j+1}$ for each $j\in\Z_+$ and $\bigcup_{j=0}^\infty K_j=\Omega$, 
\item a decreasing sequence of circle domains $M_0\supset M_1 \supset \cdots$ 
such that $\bigcap_{i=0}^\infty M_i \supset \overline \Omega$ 
and $K_j$ is $\Oscr(\mathring M_{j+1})$-convex for each $j=0,1,2,\ldots$, and
\item a sequence of holomorphic embeddings $f_j:M_j\hra \C^2$ 
such that for every $j\in\N$, the map $f_j$ approximates $f_{j-1}$ uniformly on $K_{j-1}$ 
as closely as desired and it satisfies 
\begin{equation} \label{eq:fj}
	 f_j\big(bM_j \cup \bigcup_{i\in J_j} D_i\big) \subset \C^2\setminus j\overline \B 	
	\ \ \ \text{and} \ \ \
	f_j(M_j\setminus \mathring K_{j-1}) \subset \C^2\setminus (j-1)\overline \B.
\end{equation}
\end{enumerate}
Here, $J_j\subset \N$ is such that $\mathring M_j \setminus \bigcup_{i\in J_j} D_i =\Omega$.
The second condition in \eqref{eq:fj} gives
\begin{equation}\label{eq:fj1}
	f_j(K_{j} \setminus \mathring K_{j-1}) 
	\subset f_j(M_j\setminus \mathring K_{j-1})
	\subset  \C^2\setminus (j-1)\overline \B
	\quad\text{for all}\ \ j=1,2,\ldots.
\end{equation}

Assuming as we may that the approximation of $f_{j-1}$ by $f_j$ is sufficiently close at 
every step, these conditions clearly imply that the sequence $f_j$ converges uniformly on 
compacts in $\Omega$ to a proper holomorphic embedding 
$f=\lim_{j\to\infty} f_j :\Omega\hra\C^2$.
\end{proof}

%
%
\begin{remark}
To prove Theorem \ref{th:countably-RS} in this special case, 
we modify the above construction so that
for each $j=1,2,\ldots$ we first pick a point $\alpha_j \in \Omega\setminus K_{j-1}$
and then choose the next embedding $f_{j}:M_{j}\hra\C^2$ so that it approximates 
$f_{j-1}$ on $K_{j-1}$, it satisfies 
\[
	f_{j}(\alpha_j)=\beta_j\in B
	\ \ \ \text{and}\ \ \ 
	f_j\big(bM_j \cup \bigcup_{i\in J_j} D_i\big) \subset \C^2\setminus j\overline \B,
\]
and $f_j$ agrees with $f_{j-1}$ at the previously chosen
points $\alpha_1,\ldots,\alpha_{j-1}\in K_{j-1}$ so that $f_j(\alpha_i)=\beta_i\in B$ 
holds for $i=1,\ldots,j$. We then pick the next compact, smoothly bounded, 
$\Oscr(\Omega)$-convex set $K_j\subset \Omega$ such that 
\begin{equation}\label{eq:fjbis}
	K_{j-1}\cup\{\alpha_j\}\subset \mathring K_j\ \ \text{and}\ \  
	f_j(bK_j) \subset \C^2\setminus j\overline\B.
\end{equation}
(However, we cannot fulfil condition \eqref{eq:fj1} due to the
interpolation condition $f_{j}(\alpha_j)=\beta_j$, since there is no assumption on 
the set $B=\{\beta_j\}\subset\C^2$.)
%
%
%
By choosing the set $K_j\subset \Omega$ big enough at every step to ensure
that $\bigcup_{j=1}^\infty K_j=\Omega$, condition \eqref{eq:fjbis} ensures
that the limit holomorphic embedding $f=\lim_{j\to\infty} f_j :\Omega\hra\C^2$
is almost proper. 

The general case of Theorem \ref{th:countably-RS} is proved in the following section.
\end{remark}

%
%
\begin{remark}
A geometric disc in a Riemann surface $R$ is the image of a round
disc in the universal covering space $\wt R\in \{\CP^1,\C, \D\}$ of $R$.
A circle domain in $R$ is a domain all of whose complementary
connected components are closed geometric disks and points (punctures). 
By He and Schramm \cite[Theorem 0.2]{HeSchramm1993}, every open
Riemann surface with finite genus and at most countably many ends
is conformally equivalent to a circle domain $\Omega$ in a compact Riemann surface 
$R$. If Lemma \ref{lem:mainbis} can be proved for such domains without punctures,
it would follow that any such domain embeds properly holomorphically in $\C^2$.
\end{remark}


\section{Proof of Theorem \ref{th:countably-RS}}\label{sec:countably-RS}

We shall need the following generalization of \cite[Lemma 2.2]{ForstnericWold2013}
to domains in an arbitrary compact Riemann surface $R$.

%
%
%
%
\begin{lemma}\label{lem:surrounding}
Assume that $R$ is a compact Riemann surface of genus $\nu$ 
and $\Omega$ is a connected open domain in $R$ of the same genus $\nu$. 
Given a closed set $L$ in $R$ which is a union of connected components of 
$R\setminus \Omega$ and an open set $V \subset R$ containing $L$,
there exist finitely many pairwise disjoint, smoothly bounded closed discs
$\overline{\Delta}_i\subset V$ $(i=1,\ldots,m)$ such that 
\begin{equation}\label{eq:surround}
	L\subset \bigcup_{i=1}^m \Delta_i\quad \text{and}\quad
	\bigcup_{i=1}^m b\Delta_i \subset \Omega.
\end{equation}
\end{lemma}

\begin{proof}
Let $K_1 \subset K_2\subset \cdots \subset \bigcup_{j=1}^\infty K_j=\Omega$
be an exhaustion by smoothly bounded compact connected sets
with $K_j\subset \mathring K_{j+1}$ for all $j\in\N$. By choosing $K_1$ big enough,
every set $K_j$ has the same genus $\nu$ as $R$, and hence 
$R\setminus K_j$ is the union of finitely many open discs 
$\mathcal U_j=\{U^j_1,\ldots,U^j_{m(j)}\}$ with pairwise disjoint
closures. We claim that for $j$ large enough there are discs 
$\Delta_1,\ldots, \Delta_m \in \mathcal U_j$ satisfying \eqref{eq:surround}. 
Indeed, if this is not the case, there is a decreasing sequence of closed discs 
$U^j_{k(j)}\supset U^{j+1}_{k(j+1)}$ such that $U^j_{k(j)}\cap L \neq\varnothing$ 
and $\overline{U^j_{k(j)}}\cap (R\setminus V)\neq\varnothing$ for each $j$; but then 
$\bigcap_{j=1}^\infty\overline{U^{j}_{k(j)}}$ would be a complementary 
component of $\Omega$ which is contained in $L$ and intersects $R\setminus V$,
a contradiction.  
\end{proof}

Given a compact subset $L\subset R\setminus \Omega$ and open smoothly bounded 
discs $\Delta_i\subset R$ $(i=1,\ldots,m)$ with pairwise disjoint closures satisfying 
\eqref{eq:surround}, we shall say that the set
\[
	\Gamma=\bigcup_{i=1}^m b\Delta_i\subset\Omega
\] 
is a {\em surrounding system} for $L$, or simply that $\Gamma$ surrounds $L$. The set 
\begin{equation}\label{eq:c(Gamma)}
	c(\Gamma)=\Omega\setminus \bigcup_{i=1}^m \overline \Delta_i
\end{equation}
is called the {\em core component} of $\Omega\setminus\Gamma$.
Note that if $\Gamma \subset\Omega$ surrounds $L$ and 
$\delta:[0,1)\to \Omega$ is a path such that $\delta(0)\in c(\Gamma)$ and $\delta$ 
has a limit point in $L$, then $\delta([0,1))\cap\Gamma\neq\varnothing$.

%
%
\begin{proof}[Proof of Theorem \ref{th:countably-RS}]
Let $R$ and $\Omega=R\setminus \bigcup_{i=0}^\infty D_i$ be as in the theorem, 
so $D_i$ are closed pairwise disjoint closed discs. Note that $\Omega$ 
has the same topological genus $\nu$ as $R$. Set 
\begin{equation}\label{eq:Mj}
	M_j= R \setminus \bigcup_{i=0}^{j}\mathring D_i\quad\text{for}\ \ j=0,1,\ldots.
\end{equation}
For all $j\ge 0$ we have that $\Omega\subset\mathring M_j$, 
$M_{j+1}\cup \mathring D_{j+1}=M_j$, $bM_{j+1}=bM_j\cup bD_{j+1}$, and 
$
	\Omega=\bigcap_{j=0}^\infty \mathring M_j.
$
By a {\em surrounding system} $\Gamma\subset \Omega$ for $M_j$, we shall mean a 
surrounding system for $R\setminus \mathring M_j=\bigcup_{i=0}^j D_i$.
Note that if $\delta:[0,1)\to\Omega$ is a path such that $\delta(0)\in c(\Gamma)$ 
(see \eqref{eq:c(Gamma)}) and $\delta([0,1))$ has a limit point in 
$bM_j=\bigcup_{i=0}^j bD_i$, then $\delta([0,1))$ intersects $\Gamma$.

Let $k\in\{0,1,\ldots\}$, $f:M_k\hra \C^2$, $K\subset \Omega$, 
$\epsilon>0$, and $B\subset\C^2$ be as in the statement of the theorem. 
We shall assume without loss of generality that the set $B=\{\beta_1,\beta_2,\ldots\}$ 
is infinite, and for simplicity of notation we assume that $k=0$ 
(the same argument will apply in the general case). 
Set $f_0=f:M_0=R\setminus\mathring M_0\hra \C^2$. 
Fix points $c_1,\ldots,c_l\in \Omega$ at which we wish to interpolate (see condition (ii) in 
the theorem). We may assume without loss of generality that 
$f_0(\{c_1,\ldots,c_l\})\cap B=\varnothing$. 
Using Lemma \ref{lem:surrounding} we find a smooth Jordan curve 
$\Gamma_0\subset\Omega$ surrounding the disc $D_0=R\setminus \mathring M_0$ 
such that $\Gamma_0\cap\{c_1,\ldots,c_l\}=\varnothing$
and $f_0$ vanishes nowhere on $\Gamma_0$; the last condition is easily arranged
by a small deformation. Then, choose a smoothly bounded compact 
connected domain $K_0\subset \Omega$ with genus $\nu$ such that 
$K\cup \{c_1,\ldots, c_l\}\cup \Gamma_0 \subset \mathring K_0$. 
Pick a point $\alpha_0\in\mathring K_0\setminus f_0^{-1}(B)$ 
and set $\beta_0=f_0(\alpha_0)$. We also let 
$K_{-1}=\varnothing$ and $\epsilon_0=\epsilon/2$.

We shall inductively construct a sequence of tuples 
$T_j=\{f_j,K_j,\Gamma_j,\epsilon_j,\alpha_j\}$, $j\in\N$, where
\begin{itemize}
\item $f_j: M_j\hra\C^2$ is a holomorphic embedding,
\smallskip
\item $K_j\subset \Omega$ is a smoothly bounded compact connected 
domain of genus $\nu$,\smallskip
\item $\Gamma_j\subset \mathring K_j$ is a surrounding system for $M_j$,
\smallskip
\item $\epsilon_j>0$ is a number, and
\smallskip
\item $\alpha_j\in \mathring K_j$ is a point,
\end{itemize}
such that
\begin{equation}\label{eq:exhaustion-u}
	\bigcup_{j=0}^\infty K_j=\Omega
\end{equation}
and the following conditions hold for all $j=1,2,\dots$: 
\begin{enumerate}[(1$_j$)]
\item $K_{j-1}\subset\mathring K_j$.
\smallskip
\item $\Gamma_j \cup\{\alpha_j\} \subset \mathring K_j\setminus K_{j-1}$ 
and $K_{j-1} \cup\{\alpha_j\} \subset c(\Gamma_j)$ (see \eqref{eq:c(Gamma)}). 
\smallskip
\item $\epsilon_j<\epsilon_{j-1}/2$ and every holomorphic map $\varphi: \Omega\to\C^2$ 
with $|\varphi-f_{j-1}|<2\epsilon_j$ on $K_{j-1}$ is an embedding on $K_{j-2}$.
(Recall that $K_{-1}=\varnothing$.) 
\smallskip
\item $\sup_{x\in K_{j-1}}|f_j(x)-f_{j-1}(x)|<\epsilon_j$.
\smallskip
\item $f_j(\alpha_i)=\beta_i$ for $i=0,1,\ldots,j$.
\smallskip
\item $f_j$ agrees with $f_{j-1}$ to a given order at $c_i\in K_0$ for $i=1,\ldots,l$.
\smallskip
\item $f_j(\Gamma_j)\cap j\overline\B=\varnothing$.
\end{enumerate}
The basis of the induction is given by 
$T_0=\{f_0,K_0,\Gamma_0,\epsilon_0,\alpha_0\}$. It meets conditions (1$_0$), (5$_0$), 
and (7$_0$), while the remaining conditions are void for $j=0$. 
For the inductive step, assume that we have tuples $T_0,\ldots,T_{j-1}$ satisfying the 
required conditions for some $j\ge 1$, and let us construct $T_j$. 
Choose $\epsilon_j>0$ so small that (3$_j$) holds. 
Next, choose a compact set $K_{j-1}'\subset \mathring M_j$ 
without holes in $\mathring M_j$ such that $K_{j-1}\subset\mathring K_{j-1}'$. 
(The set $K_{j-1}'$ need not be contained in $\Omega$.) 
Pick a point $\alpha_j\in \Omega\setminus K_{j-1}'$
(this set is nonempty since $K_{j-1}'\subset\mathring M_j$ is compact 
while $\Omega$ is not relatively compact in $\mathring M_j$). 
Lemma \ref{lem:main} applied to the embedding $f_{j-1}|_{M_j}:M_j\hra\C^2$, 
the compact set $K_{j-1}'\subset\mathring M_j$, the 
singletons $\{\alpha_j\}\subset\mathring M_j\setminus K_{j-1}'$ and 
$\{\beta_j\}\subset \C^2$, 
and the number $\epsilon_j>0$ furnishes a holomorphic embedding 
$f_j:M_j\hra\C^2$ satisfying (4$_j$)--(6$_j$) and 
\[ 
	f_j(bM_j)\cap j\overline\B=\varnothing.
\] 
Hence, there is an open set $V\subset R$ containing 
$R\setminus \mathring M_j=\bigcup_{i=0}^j D_i$ such that 
\[
	(K_{j-1}\cup \{\alpha_j\}) \cap \overline V=\varnothing
	\quad \text{and} \quad
	|f_j|>j\ \text{on}\ V\cap M_j.
\]
By Lemma \ref{lem:surrounding} there is a surrounding system 
$
	\Gamma_j=\bigcup_{i=0}^j \gamma_i\subset \Omega\cap V
$ 
for $M_j$ such that $\gamma_i=b\Delta_i$, where $\Delta_i\subset V$ is a 
disc containing $D_i$ and the closed discs $\overline \Delta_i$ for $i=0,1,\ldots, j$
are pairwise disjoint. Hence, $\Gamma_j\cap K_{j-1}=\varnothing$, 
$K_{j-1}\cup\{\alpha_j\}\subset c(\Gamma_j)$ (this is the second condition in (2$_j$)), 
and (7$_j$) holds. Finally, choose any smoothly bounded compact connected domain 
$K_j\subset \Omega$ containing $K_{j-1}\cup \{\alpha_j\} \cup \Gamma_j$ in its interior.
Hence, $K_j$ is of genus $\nu$ (the same as the genus of $R$) and conditions (1$_j$) and (2$_j$) hold. 
The induction may now proceed. Note that condition \eqref{eq:exhaustion-u} 
can be fulfilled since we may choose $K_j\subset \Omega$ as large as desired at each step.

As in the proof of Theorem \ref{th:dense}, there is a limit map
$
	F=\lim_{j\to\infty}f_j:\bigcup_{j=0}^\infty K_j=\Omega\hra\C^2
$
which is an injective holomorphic immersion and satisfies conditions (i), (ii), and (iii) in the 
statement of the theorem; note that $B=F(\{\alpha_j: j\in\N\})\subset F(\Omega)$. 
%
%
Finally, conditions (2$_j$)--(4$_j$) and (7$_j$) guarantee that $\inf_{x\in \Gamma_j=b(c(\Gamma_j))}|F(x)|>j-\epsilon$ for all $j\in\N$. Since $c(\Gamma_1)\Subset c(\Gamma_2)\Subset \cdots \subset \bigcup_{j\in\N}c(\Gamma_j)=\Omega$
is an exhaustion of $\Omega$ by connected, smoothly bounded compact domains in view of \eqref{eq:exhaustion-u}, (1$_j$), and (2$_j$), this inequality shows that the map $F:\Omega\to\C^2$ is almost proper. 
\end{proof}


\section{Proof of Theorem \ref{th:Cantor}}\label{sec:Cantor}

We begin by recalling a construction of a Cantor set in a  
domain $\Omega_0\subset\C$. 
In the first step, we choose a smoothly bounded compact convex domain 
$\Delta_0\subset \Omega_0$. Removing from $\Delta_0$ an open 
neighbourhood $\Upsilon_0$ of the vertical straight line segment divides 
$\Delta_0$ in two smoothly bounded compact convex subsets $\Delta_0^1$ 
and $\Delta_0^2$ of the same width. Finally, for $j=1,2$ we remove from $\Delta_0^j$ 
an open neighborhood $\Upsilon_0^j$ of the horizontal straight line segment 
dividing $\Delta_0^j$ in two convex subsets of the same height, making sure that the two
connected components of $\Delta_0^j\setminus \Upsilon_0^j$ are smoothly bounded 
compact convex domains. This gives a compact set 
\begin{equation}\label{eq:Omega1}
	\Omega_1=\Delta_0\setminus 
	(\Upsilon_0\cup\Upsilon_0^1\cup\Upsilon_0^2)
	\subset \mathring \Omega_0
\end{equation}
which is the union of four pairwise disjoint, smoothly bounded compact convex  domains
$\Omega_1^j$, $j=1,\ldots,4$.
In the second step, we repeat the same procedure for each convex compact domain 
$\Omega_1^j$ from the first generation, thereby getting four pairwise disjoint smoothly 
bounded compact convex domains in its interior. This gives a compact set 
$\Omega_2\subset\mathring \Omega_1$ which is the union of sixteen smoothly bounded 
compact convex domains. Continuing inductively, we obtain a decreasing sequence 
of smoothly bounded compact domains
\begin{equation}\label{eq:C}
	\Omega_1\Supset\Omega_2\Supset \Omega_3 \Supset\cdots
\end{equation}
such that for each $i\ge 1$ the domain 
$\Omega_i$ consists of $4^i$ pairwise disjoint smoothly bounded 
compact convex domains. 
The intersection
\begin{equation}\label{eq:C=cap}
	C = \bigcap_{i=1}^\infty \Omega_i\subset \Omega_0
\end{equation}
is then a Cantor set in $\C$. Moreover, choosing the separating neighbourhoods sufficiently 
small at each step of the construction, we may ensure that $\mu(C)>\mu(\Delta_0)-\delta$ 
for any given $\delta>0$, where $\mu$ denotes the $2$-dimensional Lebesgue measure 
on $\C$.

We now explain the proof of Theorem \ref{th:Cantor}. 
Let $R$ be a compact Riemann surface and $B\subset \C^2$ be a countable set. 
Assume that $B=\{\beta_1,\beta_2,\ldots\}$ is infinite and $0\notin B$.
Let $\Omega_0$ be a smoothly bounded compact convex domain in a holomorphic 
coordinate chart on $R$ such that there is a holomorphic embedding 
$f_0: R\setminus\mathring \Omega_0 \hra \C^2$. 
Such a set and embedding $f_0$ always exists; 
we may for instance choose $\Omega$ to be the complement of a small open neighbourhood 
of the curves in a suitable homology basis of $R$. Fix $\epsilon_0>0$, set 
$K_0=R\setminus\mathring \Omega_0$ and $K_{-1}=\varnothing$, 
and assume without loss of generality that there is a number $r_0>0$ satisfying
\begin{equation}\label{eq:inff0-C}
	f_0(K_0)\cap r_0\overline\B=\varnothing.
\end{equation}
(Cf.\ \eqref{eq:inff0}; recall that $\B$ is the unit ball in $\C^2$.) 
Choose a point $\alpha_0\in \mathring K_0\setminus f_0^{-1}(B)$ and 
set $\beta_0=f_0(\alpha_0)$. Also choose any sequence $0<r_0<r_1<r_2<\cdots$ with 
$\lim_{j\to\infty}r_j=+\infty$.

Let $\Delta_0\subset\mathring \Omega_0$ be a smoothly bounded compact convex domain 
so large that $f_0$ extends to a holomorphic embedding 
$f_0: R\setminus \mathring\Delta_0\hra\C^2$ satisfying 
$f_0(R\setminus\mathring\Delta_0)\cap r_0\overline\B=\varnothing$. 
Choose a positive number $\epsilon_1<\epsilon_0/2$. An application of Mergelyan theorem 
in two steps (see \cite[Theorem 16]{FornaessForstnericWold2020})
furnishes a compact set $\Omega_1$ as in \eqref{eq:Omega1}, 
consisting of four pairwise disjoint smoothly bounded compact convex domains, 
and a holomorphic embedding $\tilde f_0: K_1=R\setminus\mathring\Omega_1\hra \C^2$ 
such that
\begin{equation}\label{eq:tildef0}
	\tilde f_0(K_1\setminus \mathring K_0)\cap r_0\overline\B=\varnothing	
	\quad\text{and}\quad\sup_{x\in K_0}|\tilde f_0(x)-f_0(x)|<\epsilon_1/2.
\end{equation}
Indeed, we first extend $f_0$ to a smooth embedding 
$f_0: (R\setminus\mathring\Delta_0)\cup E_0\hra\C^2$, where $E_0$ is the vertical straight 
line segment dividing $\Delta_0$ in two convex subsets of the same width, such that 
$f_0(E_0)\cap r_0\overline\B=\varnothing$. Note that $E_0$ intersects 
$R\setminus \mathring\Delta_0$ only at its endpoints and the intersections are transverse. 
By Mergelyan theorem (see \cite[Theorem 16]{FornaessForstnericWold2020})
we can then approximate $f_0$ in the $\Cscr^1$ topology on 
$(R\setminus\mathring \Delta_0)\cup E_0$ by a holomorphic embedding 
$f_0': (R\setminus\mathring\Delta_0)\cup\overline\Upsilon_0\hra\C^2$, where 
$\Upsilon_0$ is a neighborhood of $E_0$ in $\Delta_0$ as explained above, 
such that 
\[
	f_0'((\Omega_0\setminus\mathring\Delta_0)\cup\overline\Upsilon_0)\cap 
	r_0\overline\B=\varnothing. 
\]
We then repeat the process simultaneously in the two components $\Delta_0^1$
and $\Delta_0^2$ of 
$\Delta_0\setminus\Upsilon_0$: we suitably extend $f_0'$ to $E_0^1\cup E_0^2$, 
where $E_0^j$, $j=1,2$, is the horizontal straight line segment dividing $\Delta_0^j$ in 
two convex subsets of the same height, and apply Mergelyan theorem to approximate $f_0'$
in the $\Cscr^1$ topology on 
$(R\setminus\mathring \Delta_0)\cup\overline\Upsilon_0\cup E_0^1\cup E_0^2$ 
by a holomorphic embedding $\tilde f_0: \Omega_1\hra\C^2$
which satisfies the conditions in \eqref{eq:tildef0} for 
$K_1=R\setminus \mathring \Omega_1$.
Note that $K_0\subset\mathring K_1$ and choose a point 
$\alpha_1\in \mathring K_1\setminus K_0$. Arguing as in the proof of Theorem \ref{th:dense}, 
we may use Lemma \ref{lem:main} to obtain a holomorphic embedding $f_1: K_1\hra\C^2$ 
satisfying the following conditions:
\begin{enumerate}[\rm (a)]
\item $\sup_{x\in K_0}|f_1(x)-\tilde f_0(x)|<\epsilon_1/2$. Hence, 
$\sup_{x\in K_0}|f_1(x)-f_0(x)|<\epsilon_1$ by \eqref{eq:tildef0}.
\smallskip
\item $f_1(\alpha_i)=\beta_i$ for $i=0,1$.
\smallskip
\item $f_1(b K_1)\cap r_1\overline\B=\varnothing$.
\smallskip
\item $f_1(K_1\setminus\mathring K_0)\cap \min\{r_0,|\beta_1|/2\} \overline\B=\varnothing$.
\end{enumerate}

We repeat this procedure inductively, following the recursive construction of a Cantor set $C$ 
in $\Omega$ described above; see \eqref{eq:C} and \eqref{eq:C=cap}. In this way, we  
construct a sequence of tuples $T_j=\{f_j,K_j,\epsilon_j,\alpha_j\}$, $j\in\N$, where
\begin{itemize}
\item $K_j=R\setminus \Omega_j$, where $\Omega_j$ is a domain in $\Omega_0$ 
consisting of $4^j$ pairwise disjoint smoothly bounded compact convex domains,
\smallskip
\item $f_j: K_j\hra\C^2$ is a holomorphic embedding,
\smallskip
\item $\epsilon_j>0$ is a number, and
\smallskip
\item $\alpha_j\in \mathring K_j\setminus K_{j-1}$ is a point,
\end{itemize}
such that $T_j$ satisfies $K_{j-1}\subset\mathring K_j$ and conditions (2$_j$)--(4$_j$), (6$_j$), 
and (7$_j$) in the proof of Theorem \ref{th:dense} for all $j\in\N$ 
(with $\mathring M$ replaced by $K_{j-1}$ in (3$_j$)), and we have that
\[
	C = R \setminus \bigcup_{j\ge 0} K_j =\bigcap_{j\ge 0}\Omega_j
\]
is a Cantor set in $R$. 
As in the proof of Theorem \ref{th:dense}, there is a limit map
\[
	F=\lim_{j\to\infty} f_j : R \setminus C=\bigcup_{j\ge 0} K_j\to\C^2
\]
which is an almost proper injective holomorphic immersion and satisfies 
\[
	B=F(\{\alpha_j: j\ge 1\})\subset F(R\setminus C).
\]
Moreover, the map $F: R\setminus C\to\C^2$ is proper, and hence a proper holomorphic embedding, if the given set $B$ is closed in $\C^2$ and discrete. 
This concludes the proof of Theorem \ref{th:Cantor}.


\section{Proof of Theorem \ref{th:pseudoconvex}}\label{sec:pseudoconvex}

Let $S$ be an open Riemann surface. Fix an exhaustion
\begin{equation}\label{eq:Sj}
	S_0\Subset S_1\Subset S_2\Subset \cdots \Subset \bigcup_{j\ge 0} S_j=S
\end{equation}
of $S$ by connected, smoothly bounded, compact domains without holes in $S$ 
such that $S_0$ is a closed disc and the Euler characteristic of 
$S_j\setminus\mathring S_{j-1}$ equals $0$ or $-1$ for all $j\ge 1$ 
(see \cite[Lemma 4.2]{AlarconLopez2013JGEA}).
Also set $D_{-1}=D_0=\varnothing$ and let
\begin{equation}\label{eq:Dj}
	D_1\Subset D_2\Subset\cdots\Subset \bigcup_{j\ge 1} D_j=D
\end{equation}
be an exhaustion of $D$ by smoothly bounded, polynomially convex, strongly pseudoconvex, 
Runge domains  \cite[Sec.\ 2.3]{Forstneric2017E}. Assume that $B=\{\beta_1,\beta_2,\ldots\}$ is 
infinite, set $m_0=0$, and for each $j\in\N$ denote by $m_j$ the unique integer such that  
\begin{equation}\label{eq:mj}
	\beta_j\in \overline D_{m_j+1}\setminus \overline D_{m_j}.
\end{equation}
Fix $\epsilon_0>0$ and set $K_0=S_0$ and $K_{-1}=\varnothing$. Let $f_0: K_0\hra D$ be a 
holomorphic embedding  with $f_0(K_0)\subset D_1$; recall that $K_0$ is a disc. 
Choose $\alpha_0\in\mathring K_0$ such that $f_0(\alpha_0)\notin B$ and set 
$\beta_0=f_0(\alpha_0)\in D_1$.  We shall construct a sequence 
$X_j=\{K_j,f_j,\epsilon_j,\alpha_j\}$, $j\in\N$, where
\begin{itemize}
\item $K_j\subset S$ is a connected, smoothly bounded, compact domain without holes in $S$,
\smallskip
\item $f_j: K_j\hra D$ is a holomorphic embedding,
\smallskip
\item $\epsilon_j>0$ is a positive number, and
\smallskip
\item $\alpha_j\in \mathring K_j\setminus K_{j-1}$ is a point,
\end{itemize}
such that the following conditions hold for all $j\in\N$:
\begin{enumerate}[(1$_j$)]
\item $K_{j-1}\Subset K_j\subset S_j$ and $K_j$ is 
%
%
diffeotopic to $S_j$.
\smallskip
\item $\sup_{x\in K_{j-1}}|f_j(x)-f_{j-1}(x)|<\epsilon_j$.
\smallskip
\item $\epsilon_j<\epsilon_{j-1}/2$ and every holomorphic map $\varphi: K_{j-1}\to\C^2$ 
with $|\varphi-f_{j-1}|<2\epsilon_j$ everywhere on $K_{j-1}$ is an embedding on $K_{j-2}$.
\smallskip
\item $f_j(\alpha_i)=\beta_i$ for all $i\in\{0,\ldots,j\}$.
\smallskip
\item $f_j(bK_i)\cap \overline D_i=\varnothing$ for all $i\in\{0,\ldots,j\}$.
\smallskip
\item $f_j(K_j\setminus\mathring K_{j-1})\cap \overline D_{j-1}\cap \overline D_{m_j}=\varnothing$.
\smallskip
\item $\length(f_j\circ\gamma)>1$ for every path $\gamma:[0,1]\to K_j$ with $\gamma(0)\in K_{j-1}$ and $\gamma(1)\in bK_j$.
\end{enumerate}

Assume for a moment that such a sequence exists. Conditions \eqref{eq:Sj} and (1$_j$) 
imply that
\begin{equation}\label{eq:M=S}
	M=\bigcup_{j\in\N} K_j\subset S
\end{equation}
is a domain in $S$ that is 
%
%
diffeotopic to $S$. 
In particular, there is a complex structure $J$ on $S$ such that the open Riemann surface 
$(S,J)$ is biholomorphic to $M$. (For details in a similar setting, see 
\cite[proof of Theorem 1.4 (b) and Corollary 1.5]{AlarconDrinovecForstnericLopez2015PLMS}.)
By (1$_j$), (2$_j$), (3$_j$), and the maximum principle, there is a limit map 
\[
	F=\lim_{j\to\infty}f_j: M\to D,
\]
which is an injective holomorphic immersion and satisfies
\begin{equation}\label{eq:2ej}
	\sup_{x\in K_{j-1}}|F(x)-f_{j-1}(x)|<2\epsilon_j,\quad j\in\N.
\end{equation}
We claim that if each $\epsilon_j>0$ is chosen sufficiently small then $F$ satisfies 
the conclusion of Theorem \ref{th:pseudoconvex}. Indeed, if every $\epsilon_j>0$ is small enough 
then \eqref{eq:M=S}, \eqref{eq:2ej}, and conditions (1$_j$) and (7$_j$) guarantee 
that $\length (F\circ\gamma)=+\infty$ for 
every proper path $\gamma: [0,1)\to M$, and hence $F$ is complete. Likewise, \eqref{eq:2ej} 
and conditions (5$_j$) imply that $F(bK_j)\cap \overline D_j=\varnothing$ for all 
$j\in\N$ whenever the $\epsilon_j$'s are sufficiently small, and hence $F: M\to D$ is an 
almost proper map in view of \eqref{eq:Dj}, \eqref{eq:M=S}, and conditions (1$_j$). 
Indeed, if $Q\subset D$ is compact and we take $m\in\N$ so large that $Q\subset D_m$, 
then $F^{-1}(Q)\cap bK_j=\varnothing$ for every $j>m$, hence all components of 
$F^{-1}(Q)$ are compact in $M$. Conditions (4$_j$) give 
that $B=\{\beta_1,\beta_2,\ldots\}\subset F(M)$. Finally, if each $\epsilon_j>0$ 
is chosen sufficiently small then 
$F(K_j\setminus\mathring K_{j-1})\cap \overline D_{j-1}\cap \overline D_{m_j}=\varnothing$ for all 
$j\in\N$ by \eqref{eq:2ej} and conditions (6$_j$), while if the given set $B\subset D$ is closed 
and discrete then $\lim_{j\to\infty}\min\{j-1,m_j\}=+\infty$; see \eqref{eq:Dj} and \eqref{eq:mj}. 
These conditions imply that $F: M\hra D$ is a proper map, and hence a proper holomorphic  
embedding, provided that $B$ is closed in $D$ and discrete. Therefore, $F$ satisfies the 
conclusion of the theorem.

To complete the proof, it remains to explain the induction. The basis is given by the tuple 
$X_0=\{K_0,f_0,\epsilon_0,\alpha_0\}$; it satisfies (1$_0$), (4$_0$), (5$_0$), and (6$_0$), 
while the remaining conditions are void for $j=0$. Fix $j\in\N$ and assume that we have 
a tuple $X_{j-1}=\{K_{j-1},f_{j-1},\epsilon_{j-1},\alpha_{j-1}\}$ fulfilling conditions
(1$_{j-1}$), (4$_{j-1}$), (5$_{j-1}$), and (6$_{j-1}$). We distinguish two cases. 

\medskip\noindent
{\em Case 1: The Euler characteristic of $S_j\setminus\mathring S_{j-1}$ is $0$.} In 
this case, \eqref{eq:Sj} and (1$_{j-1}$) imply that $K_{j-1}\subset \mathring S_j$ and 
$K_{j-1}$ is a 
strong deformation retract of $S_j$. Choose an integer $d> \max\{j,m_j+1\}$ so large that 
$f_{j-1}(K_{j-1})\subset D_d$ and set $d'=\min\{j-1,m_j\}<d$. 
Assume without loss of generality that $\beta_j\notin f_{j-1}(K_{j-1})$. 
By (5$_{j-1}$) and \eqref{eq:mj} we have that 
\begin{equation}\label{eq:DdDd'}
	f_{j-1}(bK_{j-1})\cup\{\beta_j\}\subset D_d\setminus \overline D_{d'}. 
\end{equation}
Pick a point $a\in bK_{j-1}$ and attach to $K_{j-1}$ a smooth embedded arc 
$\eta\subset \mathring S_j$ such that $\eta\cap K_{j-1}=\{a\}$ and the intersection of 
$bK_{j-1}$ and $\eta$ is transverse at $a$. On the image side, 
choose a smoothly embedded arc 
\begin{equation}\label{eq:lambdaDd}
	\lambda\subset D_d\setminus \overline D_{d'}
\end{equation} 
such that $\lambda$ agrees with $f_{j-1}(\eta)$ near the endpoint $f_{j-1}(a)$, 
$\lambda\cap f_{j-1}(K_{j-1})=f_{j-1}(a)$, and the other endpoint of $\lambda$ equals $\beta_j$. 
Let $\alpha_j$ denote the other endpoint of $\eta$ and extend $f_{j-1}$ to a smooth 
diffeomorphism $\eta\to\lambda$ such that $f_{j-1}(\alpha_j)=\beta_j$.
By Mergelyan theorem (see \cite[Theorem 16]{FornaessForstnericWold2020})
we may assume in view of (5$_{j-1}$), \eqref{eq:DdDd'}, and \eqref{eq:lambdaDd} that 
there is a connected, smoothly bounded, compact domain $K\subset S$ without holes such that 
\begin{enumerate}[\rm (i)]
\item $\alpha_j\in K_{j-1}\cup\eta\Subset K\subset \mathring S_j$, 
\smallskip
\item $K_{j-1}$ is a strong deformation retract of $K$, and 
\smallskip
\item $f_{j-1}: K\hra D_d$ is a holomorphic embedding satisfying
\[
	f_{j-1}(\alpha_j)=\beta_j\quad \text{and}\quad 
	f_{j-1}(\overline{K\setminus K_{j-1}})\cap \overline D_{d'} =\varnothing.
\] 
\end{enumerate}
It follows that $K$ is a strong deformation retract of $S_j$ as well. 
By Charpentier and Kosi\'{n}ski \cite[Lemma 2.4]{CharpentierKosinski2020} 
there is a compact polynomially convex set $\Gamma\subset D_{d+1}\setminus \overline D_d$
whose connected components are holomorphically contractible (for example, convex)
such that $\overline D_d\cup\Gamma$ is polynomially convex 
and $\length(\sigma)>1$ for every path $\sigma:[0,1]\to D\setminus \Gamma$ with 
$\sigma(0)\in \overline D_d$ and $\sigma(1)\in D\setminus D_{d+1}$.
It follows that the compact set 
\[
	L=\overline D_{d'}\cup \Gamma= (\overline D_{j-1}\cap \overline D_{m_j})\cup\Gamma
\]
is also polynomially convex, and 
$f_{j-1}(\overline{K\setminus K_{j-1}})\cap L =\varnothing$ by (iii). Lemma \ref{lem:main} furnishes a holomorphic embedding $f_j: K\hra \C^2$ 
satisfying the following conditions:
\begin{enumerate}[\rm (a)]
\item $f_j(bK)\cap \overline D_{d+1}=\varnothing$.
\smallskip
\item $f_j(\overline{K\setminus K_{j-1}})\cap L=\varnothing$.
\smallskip
\item $\sup_{x\in K_{j-1}\cup\eta}|f_j(x)-f_{j-1}(x)|<\epsilon_j$ (note that $K_{j-1}\cup\eta$ 
has no holes in $\mathring K$).
\smallskip
\item $f_j(\alpha_i)=f_{j-1}(\alpha_i)=\beta_i$ for all $i\in\{0,\ldots,j\}$ (see (iii) and (4$_{j-1}$)).\smallskip
\item $f_j(bK_i)\cap\overline D_i=\varnothing$ for all $i\in\{0,\ldots,j-1\}$ (see (5$_{j-1}$)).
\end{enumerate}
Here, $\epsilon_j>0$ is so small that condition (3$_j$) holds and 
\begin{equation}\label{eq:Dd}
	f_j(K_{j-1}\cup\eta)\subset D_d
\end{equation}
(see (iii) and (c)). Note that \eqref{eq:Dd} ensures that the point $\alpha_j\in\eta$ lies in the 
connected component of $f_j^{-1}(D_d)$ containing $K_{j-1}$. This, 
the maximum principle, and conditions (i) and (a) 
guarantee the existence of a connected, smoothly bounded, compact domain 
$K_j\subset \mathring K$ without holes in $S$ satisfying (1$_j$),
\begin{equation}\label{eq:Kj}
	\alpha_j\in \mathring K_j,\quad f_j(bK_j)\cap\overline D_{d+1}
	=\varnothing,\quad \text{and}\quad f_j(K_j)\subset D.
\end{equation}
Condition (c) implies (2$_j$); (d) ensures (4$_j$); (e), \eqref{eq:Kj}, and $d>j$ give (5$_j$); 
(b) implies (6$_j$); and (b), \eqref{eq:Dd}, \eqref{eq:Kj}, and the properties of $\Gamma$ 
ensure (7$_j$). This closes the induction in this case.

\medskip\noindent
{\em Case 2: The Euler characteristic of $S_j\setminus\mathring S_{j-1}$ equals $-1$.} 
In this case, there is a smooth Jordan arc $E\subset \mathring S_j\setminus \mathring K_{j-1}$, 
transversely attached with its two endpoints to $bK_{j-1}$ and otherwise disjoint from $K_{j-1}$, 
such that $K_{j-1}\cup E$ is a strong deformation retract of $S_j$. Given $\epsilon>0$, an 
application of Mergelyan theorem (see \cite[Theorem 16]{FornaessForstnericWold2020}) 
furnishes a connected, smoothly bounded, compact domain $K\subset \mathring S_j$, 
without holes in $S$, such that $K_{j-1}\cup E\subset \mathring K$ and $K$ is a strong 
deformation retract of $S_j$, and a holomorphic embedding $g: K\hra D$ such that 
$\sup_{x\in K_{j-1}}|g(x)-f_{j-1}(x)|<\epsilon$, 
$g(\alpha_i)=\beta_i$ for all $i\in\{0,\ldots,j-1\}$ (see (4$_{j-1}$)),
$g(bK_i)\cap\overline D_i=\varnothing$ for all $i\in\{0,\ldots,j-1\}$, and 
$g(K\setminus \mathring K_{j-1})\cap \overline D_{j-1}=\varnothing$ (see (5$_{j-1}$)). 
See \cite[p.\ 216, Case 1]{Alarcon2020JAM} for the details in a very similar situation. 
This reduces the proof of the inductive step to Case 1.
This closes the induction and completes the proof of Theorem \ref{th:pseudoconvex}.


\section{Proof of Theorems \ref{th:main} and \ref{th:main2}}\label{sec:new}

%
%
For simplicity of exposition we shall prove these results in the case when the open Riemann surface $S$ is a disc, say, $S=2\D=\{\zeta\in\C: |\zeta|<2\}$. In particular, the domain $M\subset S$ in Theorem \ref{th:main} must be a disc, while the one in Theorem \ref{th:main2} must be a planar domain. The general cases are seen by combining the proof in these special cases with the procedure to prescribe the topology in the proof of Theorem \ref{th:pseudoconvex}; we leave the details to interested readers.

%
%
\begin{proof}[Proof of Theorem \ref{th:main}]
Let $X\subset\C^2$ and $B=\{\beta_j\}_{j\in\N}\subset X$ be as in the statement. 
Let us assume that $S=2\D$. Set $D_0=\overline \D$ and choose a  
holomorphic embedding $f_0:D_0\to X$. Assume that $f_0(0)\notin B$ 
and set $\alpha_0=0$, $\beta_0=f(\alpha_0)$, and $D_{-1}=\varnothing$. Fix a number 
$\epsilon_0>0$. We shall inductively construct a sequence of smoothly bounded closed 
discs $D_j\subset 2\D$, points $\alpha_j\in \mathring D_j$, holomorphic embeddings 
$f_j: D_j\to X$, and numbers $\epsilon_j>0$ satisfying the following conditions for all 
$j\in\N=\{1,2,\ldots\}$:
%
%
\begin{enumerate}[\rm (i$_j$)]
\item $D_{j-1}\subset \mathring D_j$.
\item $|f_j-f_{j-1}|<\epsilon_j$ on $D_{j-1}$.
\item $f_j(\alpha_k)=\beta_k$ for all $k\in\{0,\ldots,j\}$.
\item $\dist_{f_j}(0,bD_j)>j$,
%
%
where $\dist_{f_j}$ denotes the distance function on $D_j$ associated to the Riemannian metric induced on $D_j$ by the Euclidean one in $\C^2$ via the immersion $f_j$.
\item $\epsilon_j<\epsilon_{j-1}/2$ and every holomorphic map $\varphi: D_{j-1}\to\C^2$ such that $|\varphi-f_{j-1}|<2\epsilon_j$ on $D_{j-1}$ is an embedding and satisfies $\varphi(D_{j-1})\subset X$ and $\dist_{\varphi}(0,bD_{j-1})>j-1$.
\end{enumerate}
The basis of the induction is provided by the already chosen disc $D_0=\overline\D$, 
point $\alpha_0=0\in\mathring D_0$, holomorphic embedding $f_0:D_0\to X$, and number 
$\epsilon_0>0$. They meet conditions {\rm (i$_0$)}, {\rm (iii$_0$)}, and {\rm (iv$_0$)}, 
while {\rm (ii$_0$)} and {\rm (v$_0$)} are void. 
For the inductive step, fix $j\in\N$ and assume that we have suitable objects $D_k$, 
$\alpha_k$, $f_k$, and $\epsilon_k$ satisfying {\rm (i$_k$)}, {\rm (iii$_k$)}, and {\rm (iv$_k$)} for all $k\in\{0,\ldots,j-1\}$. By {\rm (iv$_{j-1}$)} we can choose a number $\epsilon_j>0$ 
so small that {\rm (v$_j$)} is satisfied.
Reasoning as in Case 1 in the proof of Theorem \ref{th:pseudoconvex}, we may assume that there are a smoothly bounded closed disc $\Sigma\subset 2\D$ and a point $\alpha_j\in \Sigma$ such that $f_{j-1}$ extends to a holomorphic embedding $f_{j-1}:\Sigma\to X$ with
\begin{equation}\label{eq:aj-New}
	\text{$\alpha_j\in \mathring \Sigma\supset D_{j-1}$ \; and \; $f_{j-1}(\alpha_j)=\beta_j$.}
\end{equation}
Moreover, $\Sigma$ can be chosen as close to $D_{j-1}$ as desired.

Since $\Sigma$ is a disc and $f_{j-1}:\Sigma\to\C^2$ is holomorphic, 
the compact set $f_{j-1}(\Sigma)\subset X$ is polynomially convex in $\C^2$ 
%
%
by a theorem of Wermer \cite{Wermer1958}
(see also Stolzenberg \cite{Stolzenberg1966AM} and Alexander \cite{Alexander1971}), 
so it admits a basis of open neighbourhoods which 
are smoothly bounded, strongly pseudoconvex, and Runge in $\C^2$. 
(Indeed, a compact polynomially convex set $K\subset \C^n$ is the zero set of a 
smooth plurisubharmonic exhaustion function $\rho\ge 0$ on $\C^n$ 
\cite[Theorem 1.3.8]{Stout2007}, and every sublevel set $\{\rho<c\}$ for $c>0$ of such 
a function is a pseudoconvex Runge domain in $\C^n$ by 
\cite[Theorem 4.3.4]{Hormander1990}.) 
%
%
Let $U_1\Subset U_2\Subset X$ 
be a pair of such
%
%
relatively compact neighbourhoods of $f_{j-1}(\Sigma)$ in $X$. 
By \cite[Lemma 2.4]{CharpentierKosinski2020} there is a compact polynomially convex set 
$\Gamma\subset U_2\setminus\overline U_1$ whose connected components are 
holomorphically contractible such that $\overline U_1\cup \Gamma$ is polynomially convex 
and 
$\length(\gamma)>1$ for every path 
$\gamma:[0,1]\to X\setminus\Gamma$ with $\gamma(0)\in \overline U_1$ and 
$\gamma(1)\in X\setminus \overline U_2$. 
Since $D_{j-1}$ is a Runge compact 
in $\mathring\Sigma$ and $f_{j-1}(\Sigma)\cap \Gamma=\varnothing$, 
Lemma \ref{lem:main} furnishes a holomorphic embedding 
$f_j:\Sigma\hra \C^2$ satisfying the following conditions:
\begin{enumerate}[\rm (a)]
\item $f_j(b\Sigma)\cap \overline U_2=\varnothing$; 
%
%
recall that $U_2$ is compact.
\item $f_j(\Sigma\setminus\mathring D_{j-1})\cap \Gamma=\varnothing$.
%
%
\item $|f_j-f_{j-1}|<\epsilon_j$ on a smoothly bounded closed disc $\Sigma'$ with $D_{j-1}\cup\{\alpha_j\}\Subset\Sigma'\Subset \Sigma$.
\item $f_j(\alpha_k)=f_{j-1}(\alpha_k)$ for all $k\in\{0,\ldots,j\}$.
\end{enumerate}
Further, 
%
%
%
%
choosing $\epsilon_j>0$ sufficiently small, condition (c) implies that
\begin{equation}\label{eq:fjDj-1-New}
	f_j(D_{j-1})\subset f_j(\Sigma')\subset U_1;
\end{equation} 
recall that $U_1$ is an open neighbourhood of $f_{j-1}(\Sigma)$.
Since the domain 
$U_2$ is Runge in $\C^2$, 
this and conditions (a) and (c) guarantee the existence of a smoothly bounded closed disc 
$D_j\subset\mathring \Sigma\subset 2\D$, 
%
%
with $\Sigma'\subset \mathring D_j$ (this implies (i$_j$)), satisfying 
$f_j(D_j)\subset X$ and $f_j(bD_j)\cap\overline U_2=\varnothing$.
This, \eqref{eq:fjDj-1-New}, (b), and the properties of $\Gamma$ imply 
$\dist_{f_j}(D_{j-1},bD_j)>1$, while (c)=(ii$_j$) and (v$_j$) ensure that 
$\dist_{f_j}(0,bD_{j-1})>j-1$, so (iv$_j$) holds. Finally, (d), (iii$_{j-1}$), 
and \eqref{eq:aj-New} imply (iii$_j$). This closes the induction.

Note that $M=\bigcup_{j\in\N} D_j\subset 2\D$ is an open disc 
%
%
which is 
%
%
diffeotopic to $S=2\D$. By conditions (i$_j$), (ii$_j$), 
and (v$_j$), there is a limit holomorphic map $F=\lim_{j\to\infty}f_j:M\to\C^2$
such that $|F-f_{j-1}|<2\epsilon_j$ for all $j\in\N$. So, by (v$_j$), $F$ has range in $X$ and 
is a 
complete injective immersion. Finally, conditions (iii$_j$) ensure that $B\subset F(M)$.
%
%
\end{proof}

%
%
\begin{proof}[Proof of Corollary \ref{co:main1}]
Let $G$, $K$, and $\epsilon$ be as in the statement. By Mergelyan theorem (see 
\cite[Theorem 16]{FornaessForstnericWold2020}) there is a holomorphic embedding 
$f_0:\overline\D\to\C^2$ with $|f_0-G|<\epsilon/2$ on $\overline \D$, hence 
$d_{\rm H}(f_0(\overline\D),G(\overline \D))<\epsilon/2$. Let $X\subset\C^2$ be an 
$\epsilon/2$-tubular neighbourhood of $f_0(\overline\D)$. It follows that 
$d_{\rm H}(\overline X,G(\overline \D))<\epsilon$. By the proof of Theorem \ref{th:main} there 
are an open disc $M$ with $\D\Subset M\Subset 2\D$ and a 
complete injective 
holomorphic immersion $h:M\to X$ such that $h(M)$ is everywhere dense in $X$, i.e., 
$\overline{h(M)}=\overline X$. We then have that 
$d_{\rm H}(\overline{h(M)},G(\overline \D))<\epsilon$. Furthermore, an inspection 
of the proof of Theorem \ref{th:main} shows that we can choose $M$ so close to $\D$ that 
there is a holomorphic diffeomorphism $\phi:\D\to M$ satisfying $|h\circ\phi-f_0|<\epsilon/2$ 
on the compact set $K\subset\D$. It is clear that $F=h\circ\phi$ satisfies the conclusion 
of the corollary.
\end{proof}

%
%
\begin{proof}[Proof of Theorem \ref{th:main2}]
Let $B=\{\beta_j\}_{j\in\N}\subset X\subset \C^2$ be as in the statement and 
$S=2\D$. Let $D_0$, $f_0$, $\alpha_0$, $\beta_0$, $D_{-1}$, and $\epsilon_0$ be as in the proof of Theorem \ref{th:main}. Choose an exhaustion
\begin{equation}\label{eq:K0-New}
	\varnothing = K_0\Subset K_1\Subset K_2\Subset\cdots\subset\bigcup_{j\in\N}K_j=X
\end{equation}
of $X$ by compact domains. We shall inductively construct an increasing 
sequence of connected, smoothly bounded compact domains $D_j\subset 2\D$, 
as well as sequences of points $\alpha_j\in \mathring D_j$, holomorphic embeddings 
$f_j: D_j\to X$, and numbers $\epsilon_j>0$ satisfying conditions (i$_j$)--(v$_j$) 
in the proof of Theorem \ref{th:main} and also the following one for all $j\in\N=\{1,2,\ldots\}$:
%
%
\begin{enumerate}
\item[\rm (vi$_j$)] $f_j(bD_j)\cap K_j=\varnothing$.
\end{enumerate}
(Unlike in the proof of Theorem \ref{th:main}, $D_j$ need not be a disc.) 
Note that (vi$_0$) holds true. For the inductive step, fix $j\in\N$ and assume that we have 
suitable objects $D_k$, $a_k$, $f_k$, and $\epsilon_k$ satisfying {\rm (i$_k$)}, {\rm (iii$_k$)}, 
{\rm (iv$_k$)}, and (vi$_k$) for all $k\in\{0,\ldots,j-1\}$. 
%
%
Choose $\epsilon_j>0$ so small that 
{\rm (v$_j$)} holds. Reasoning as in the proof of Theorem \ref{th:main} we may assume that 
$f_{j-1}$ extends to a holomorphic embedding $f_{j-1}:\Sigma\to X$ on a connected, 
smoothly bounded compact domain 
$\Sigma \subset 2\D$  such that 
$D_{j-1}\subset \mathring \Sigma$, $D_{j-1}$ is a strong deformation retract of $\Sigma$, 
and there is a point $\alpha_j\in\mathring \Sigma$ with $f_{j-1}(\alpha_j)=\beta_j$. By a small perturbation 
of the map $f_{j-1}$ keeping the above conditions in place, we can ensure in addition that 
$f_{j-1}(\Sigma)$ is polynomially convex in $\C^2$; see the argument in the first part of the 
proof of Lemma \ref{lem:main} based on Stolzenberg's 
theorem \cite{Stolzenberg1963AM}. We can therefore choose a pair of smoothly bounded,
relatively compact, pseudoconvex domains $U_1\Subset U_2\Subset X$ 
%
%
which are Runge in $\C^2$ such that $f_{j-1}(\Sigma)\subset U_1$. 
%
%
We place a suitable labyrinth $\Gamma=\Gamma_j$ in $U_2\setminus \overline U_1$ and
choose a holomorphic embedding $f_j:\Sigma\to\C^2$ as in the proof of 
Theorem \ref{th:main} with condition (a) replaced by
\begin{enumerate}[\rm (a')]
\item $f_j(b\Sigma)\cap (K_j\cup \overline U_2)=\varnothing$.
\end{enumerate} 
In particular, conditions (b)--(d) and \eqref{eq:fjDj-1-New} in the proof of Theorem \ref{th:main} are 
satisfied
%
%
(in this case $\Sigma'$ is a smoothly bounded compact domain in $\mathring \Sigma$ containing $D_{j-1}\cup\{\alpha_j\}$ in its interior and such that $D_{j-1}$ is a strong deformation retract of $\Sigma'$). In view of {\rm (a')} and the mentioned conditions, there is a connected, smoothly
bounded compact domain $D_j\subset\mathring\Sigma\subset 2\D$ such that 
\[
\text{$D_{j-1}\cup\{\alpha_j\}\subset\mathring D_j$, \; $f_j(D_j)\subset X$, \; and \; 
$f_j(bD_j)\cap (K_j\cup \overline U_2) =\varnothing$.} 
\]
(The last condition will be the key to ensure almost properness of the limit map; compare with 
the 
condition $f_j(bD_j)\cap\overline U_2=\varnothing$ in the proof 
of Theorem \ref{th:main}. In general, since the domain $X$ need not be pseudoconvex
%
%
and Runge in $\C^2$, we cannot choose $D_j$ so that $D_{j-1}$ is a strong deformation 
retract of $D_j$; possibly $D_j$ has more boundary components than $D_{j-1}$.) 
Then, conditions (i$_j$)--(vi$_j$) hold true, which closes the induction.

Set $M=\bigcup_{j\in\N}D_j\subset 2\D$, a connected relatively compact domain in 
$\C$. By the reasoning in the proof of Theorem \ref{th:main}, there is a limit map 
$F=\lim_{j\to\infty} f_j:M\to X$ which is a complete injective holomorphic 
immersion such that $B\subset F(M)$. Finally, condition (vi$_j$) ensures that $F$ is an 
almost proper map provided that each number $\epsilon_j>0$ in the inductive process is 
chosen sufficiently small. Indeed, by such a choice (similar to that in (v$_j$)) we can manage 
to grant that $F(bD_j)\cap K_j=\varnothing$ for all $j\in\N$; see (ii$_j$) and (vi$_j$). 
%
%
This, (i$_j$), and \eqref{eq:K0-New} imply the almost properness of $F:M\to X$.
\end{proof}

%
%
%
%
\smallskip
\noindent {\bf Acknowledgements.} Alarc\'on is partially supported by the State Research Agency (AEI) via the grant no.\ PID2020-117868GB-I00 and the ``Maria de Maeztu'' Excellence Unit IMAG, reference CEX2020-001105-M, funded by MCIN/AEI/10.13039/ 501100011033/; and the Junta de Andaluc\'ia grant no. P18-FR-4049; Spain.

Forstneri\v c is supported by the 
European Union (ERC Advanced grant HPDR, 101053085) and grants P1-0291, J1-3005, and N1-0237 from ARRS, Republic of Slovenia.




\end{document}